\documentclass[titlepage,12pt]{article} 
\usepackage{hyperref}
\usepackage[usenames,dvipsnames]{xcolor}
\usepackage[title]{appendix}
\usepackage{amssymb,amsthm,amsmath} 
\usepackage[a4paper]{geometry}
\usepackage{enumitem}
\usepackage[utf8]{inputenc}

\date{}

\newcommand{\boxa}[1]{\fbox{$\displaystyle#1$}}

\newcommand{\dath}{D(A)\times H}
\newcommand{\hg}{\widehat{g}}
\newcommand{\gaph}{\mathcal{G}(\lambda_{1},\lambda_{2},H)}

\newcommand{\ep}{\varepsilon}
\renewcommand{\qed}{{\penalty 10000\mbox{$\quad\Box$}}}
\newcommand{\re}{\mathbb{R}}

\newtheorem{thm}{Theorem}[section]

\newtheorem{rmk}[thm]{Remark}
\newtheorem{prop}[thm]{Proposition}
\newtheorem{defn}[thm]{Definition}
\newtheorem{cor}[thm]{Corollary}

\newtheorem{lemma}[thm]{Lemma}

\title{An infinite dimensional Duffing-like evolution equation with linear dissipation and an asymptotically small source term}

\author{Marina Ghisi\vspace{1ex}\\ 
{\normalsize Universit\`a degli Studi di Pisa} \\
{\normalsize Dipartimento di Matematica}\\ 
{\normalsize PISA (Italy)}\\
{\normalsize e-mail: \texttt{marina.ghisi@unipi.it}}
\and
Massimo Gobbino\vspace{1ex}\\ 
{\normalsize Universit\`a degli Studi di Pisa} \\
{\normalsize Dipartimento di Ingegneria Civile e Industriale}\\ 
{\normalsize PISA (Italy)}\\  
{\normalsize e-mail: \texttt{massimo.gobbino@unipi.it}}
\and
Alain Haraux\vspace{1ex}\\ 
{\normalsize Universit\'{e} Pierre et Marie Curie} \\
{\normalsize Laboratoire Jacques-Louis Lions}\\ 
{\normalsize PARIS (France)}\\  
{\normalsize e-mail: \texttt{haraux@ann.jussieu.fr}}}

\begin{document}
\maketitle
\begin{abstract}

We consider an abstract nonlinear second order evolution equation, inspired by some models for damped oscillations of a beam subject to external loads or magnetic fields, and shaken by a transversal force. When there is no external force, the system has three stationary positions, two stable and one unstable, and all solutions are asymptotic for $t$ large to one of these stationary solutions.	

We show that this pattern extends to the case where the external force is bounded and small enough, in the sense that solutions can exhibit only three different asymptotic behaviors.
\vspace{6ex}

\noindent{\bf Mathematics Subject Classification 2010 (MSC2010):} 
35B40, 35L75, 35L90.

		
\vspace{6ex}

\noindent{\bf Key words:} Duffing equation, asymptotic behavior, dissipative hyperbolic equation, magneto-elastic oscillations.

\end{abstract}

 
\section{Introduction}

Let $H$ be a real Hilbert space, in which $|x|$ denotes the norm of an
element $x\in H$, and $\langle x,y\rangle$ denotes the scalar product
of two elements $x$ and $y$.  Let $A$ be a self-adjoint positive operator on
$H$ with dense domain $D(A)$.  

We consider some evolution problems of the following form
\begin{equation}
u''+\delta u'+k_{1}A^{2}u-k_{2}Au+k_{3}|A^{1/2}u|^{2}Au=f(t),
\label{eqn:duffing-cpx}
\end{equation}
where $\delta$, $k_{1}$, $k_{2}$, $k_{3}$ are positive constants, and $f:[0,+\infty)\to H$ is a given forcing term, with initial data
\begin{equation}
u(0)=u_{0},
\hspace{4em}
u'(0)=u_{1}.
\nonumber
\end{equation}

A  concrete example of an equation that fits in  this abstract framework is the partial differential equation
\begin{equation}
u_{tt}+\delta u_{t}+k_{1}u_{xxxx}+k_{2}u_{xx}-k_{3}\left(\int_{0}^{1}u_{x}^{2}\,dx\right)u_{xx}=f(t, x)
\label{eqn:duffing-concr}
\end{equation}
in the strip $(t,x)\in [0,+\infty)\times [0,1]$, with boundary conditions
\begin{equation}
u(t,x)=u_{xx}(t,x)=0
\qquad
\forall(t,x)\in [0,+\infty)\times \{0,1\}.
\label{eqn:bc-concr}
\end{equation}

\paragraph{\textmd{\textit{Physical models and experiments}}}

Equation (\ref{eqn:duffing-concr}) appears in~\cite{holm-mars} as a model for the motion of a beam which is buckled by an external load $k_{2}$, and shaken by a transverse displacement $f(t)$ (depending only on time, in that model). The boundary conditions (\ref{eqn:bc-concr}) correspond to ``hinged ends'', but many different choices are also possible. Equation (\ref{eqn:duffing-concr}) becomes a special case of (\ref{eqn:duffing-cpx}) if we choose $H:=L^{2}((0,1))$ and $Au=-u_{xx}$ with homogeneous Dirichlet boundary conditions.

A different physical model leading to equations of the form (\ref{eqn:duffing-concr}), althought with different boundary conditions, is the so called magneto-elastic cantilever beam described in~Figure~1 of~\cite{moon-holm}. The physical apparatus consists in a beam which is clamped vertically at the upper end, and suspended at the other end between two magnets secured to a base. The whole system is shaken by an external force transversal to the beam. 

Both systems exhibit a somewhat complex behavior. To begin with, let us consider the case without external force. When $k_{2}$ is small enough, the trivial solution $u(t)\equiv 0$ is stable. This regime corresponds to a small external load in the first model, and to a large distance from the magnets in the case of the magneto-elastic beam. When $k_{2}$ increases, the trivial solution becomes unstable, and two nontrivial equilibrium states appear. In this new regime, the effect of an external force seems to depend deeply on the size of the force itself. If the force is small enough, experiments reveal that solutions remain close to the equilibrium states of the unforced system. On the contrary, when the external force is large enough, trajectories seem to show a chaotic behavior. Describing and modelling this chaotic behavior was actually the main goal of~\cite{moon-holm,holm-mars}.

\paragraph{\textmd{\textit{Simple modes and Duffing's equation}}}

Up to changing the unknown and the operator according to the rules
$$u(t)\rightsquigarrow\alpha u(\beta t),
\qquad\qquad
A\rightsquigarrow\gamma A$$
for suitable values of $\alpha$, $\beta$, $\gamma$, we can assume that three of the four constants in (\ref{eqn:duffing-cpx}) are equal to~1. We end up, naming for simplicity the new unknown by $u$ as well,  with the equation 
\begin{equation}
u''+u'+A^{2}u-\lambda Au+|A^{1/2}u|^{2}Au=f(t).
\label{eqn:duffing}
\end{equation} with the initial conditions renamed accordingly \begin{equation}
u(0)=u_{0},
\hspace{4em}
u'(0)=u_{1}.
\label{eqn:data}
\end{equation}

Just to fix ideas, we can also assume, as in the concrete example (\ref{eqn:duffing-concr}), that $H$ admits an orthonormal basis $\{e_{n}\}$ made by eigenvectors of $A$, corresponding to an increasing sequence $\lambda_{1}<\lambda_{2}<\ldots$ of positive eigenvalues. If we restrict equation (\ref{eqn:duffing}) to the $k$-th eigenspace, we obtain an ordinary differential equation of the form
\begin{equation}
u_{k}''+u_{k}'+\lambda_{k}(\lambda_{k}-\lambda)u_{k}+\lambda_{k}^{2}u_{k}^{3}=f_{k}(t).
\label{eqn:duffing-k}
\end{equation}

Of course (\ref{eqn:duffing}) is not equivalent to the system made by (\ref{eqn:duffing-k}) as $k$ varies, because of the coupling due to the nonlinear term. Nevertheless, in the special case where both initial data and the external force are multiples of a given eigenvector $e_{k}$, equation (\ref{eqn:duffing}) reduces exactly to (\ref{eqn:duffing-k}).

Equation (\ref{eqn:duffing-k}) is known in the mathematical literature as \emph{Duffing's equation}. When there is no external force, namely $f_{k}(t)\equiv 0$, it is well-known that the behavior of solutions depend on the sign of the coefficient of $u_{k}$, or equivalently of $ \lambda_{k}-\lambda $. 
\begin{itemize}
  \item When $ \lambda < \lambda_{k}$, we are in the so-called \emph{hardening regime}, in which the trivial solution $u_{k}(t)\equiv 0$ is the unique stationary solution, and it is stable. 
  \item When $ \lambda > \lambda_{k}$, we are in the so-called \emph{softening regime}, in which the trivial solution is unstable, and (\ref{eqn:duffing-k}) has exactly two nontrivial stable equilibrium solutions
$$u(t)\equiv\pm\sqrt{\frac{\lambda-\lambda_{k}}{\lambda_{k}}}.$$
\end{itemize}

In  particular, when $\lambda<\lambda_{1}$, all projections end up in the hardening regime, and this is the case where experiments revealed stability of the trivial solution. When $\lambda\in(\lambda_{1},\lambda_{2})$, the first projection is in the softening regime, while all other projections are in the hardening regime. This is the range considered in~\cite{moon-holm,holm-mars}, and it is also the range we consider in this paper. In this range, a natural toy model for (\ref{eqn:duffing}) is the ordinary differential equation (\ref{eqn:duffing-k}) with $k=1$. This equation, with a simple variable change, can be put in the more standard form
\begin{equation}
u''+\delta u'-u+u^{3}=f(t).
\label{eqn:duffing-h}
\end{equation}

\paragraph{\textmd{\textit{Stability vs chaos}}}

In the hardening regime, namely when in (\ref{eqn:duffing-h}) the term $-u+u^{3}$ is replaced by $u+u^{3}$, it has been known for quite a while that, if $f(t)$ is globally defined and small enough in uniform norm, there exists a unique solution that is bounded on the whole line, and this solution is periodic if $f(t)$ is periodic (see~\cite {loud2}). 

On the contrary, the behavior of solutions can be quite complicated when $f(t)$ is large enough. For example, if $f(t)$ is periodic with minimal period $T$, there may exist solutions with minimal period equal to a multiple of $T$, known in the literature as ``subharmonic'' solutions (see for example the papers~\cite{F-H} and~\cite{G-H, N-S} concerning the number of subharmonic solutions for large forcing terms). The existence of subharmonic solutions with arbitrarily large periods may lead to a chaotic dynamic, with, usually, a classical transition scenario for large multiples of a given forcing term. 

This phenomenon has motivated many specialized papers in the middle of the twentieth century, where nonlinear terms more general than the cube have been considered (see for example~\cite{loud1, loud3}).

In the softening regime, in the cubic case the issue is to identify three privileged solutions playing the role of the equilibria when the forcing term is small enough in uniform norm. Partial steps in this direction were achieved in~\cite{haraux:duffing}, where the  two following results are proven assuming that $f$ is defined and bounded on the whole line.
\begin{itemize}
  \item If $f$ is small enough in $L^\infty (\re)$, equation (\ref{eqn:duffing-h}) has three special solutions, which are the generalization of the two stable solutions $u(t)\equiv\pm 1$ and of the unstable solution $u(t)\equiv 0$ of the unforced equation. When $f $ is periodic with (minimal) period $T>0$, the special solutions are $T$-periodic as well.
  \item If $f$ is small enough in $L^\infty (\re)$, and in addition $\delta\geq 2\sqrt{2}$, then all solutions to  (\ref{eqn:duffing-h}) are asymptotic to one of the special ones as $t\to +\infty$. In the $T$-periodic case, this asymptotic convergence result is enough to exclude the presence of subharmonic solutions or more chaotic behavior.
\end{itemize}

Unfortunately, the techniques of~\cite{haraux:duffing}, which have been followed in a more general context by \cite{F-G}, seem to require in an essential way that $\delta\geq 2\sqrt{2}$, which coming back to the toy model of (\ref{eqn:duffing-k}) with $k=1$ corresponds to asking that $\lambda_{1}(\lambda-\lambda_{1})$ is small enough.

\paragraph{\textmd{\textit{Our results}}} 

In this paper we consider the full equation (\ref{eqn:duffing}) in the infinite dimensional setting, again in the range $\lambda\in(\lambda_{1},\lambda_{2})$. Under smallness assumptions on the forcing term, but without any further restriction on $\lambda$, $\lambda_{1}$ and $\lambda_{2}$, we prove that all solutions remain close, as $t\to +\infty$, to one of the three stationary solutions to the unforced equation, within a distance depending on the size of the forcing term. Moreover, two solutions that are eventually close to the same stationary point are actually asymptotic to each other. Finally, out of the three possible asymptotic profiles, two are stable in the sense that the set of initial data that originate solutions converging to them is a nonempty open set; the other one corresponds to the physically irrelevant ``solution in between''.

These results are new even in the simple setting of (\ref{eqn:duffing-h}), because they imply that the asymptotic convergence result of~\cite{haraux:duffing} holds true without the technical assumption that $\delta\geq 2\sqrt{2}$.

\paragraph{\textmd{\textit{Structure of the paper}}} 

This paper is organized as follows. In section~\ref{sec:statements} we clarify the functional setting, we state a preliminary well-posedness result for (\ref{eqn:duffing}) (Theorem~\ref{thmbibl:wp}), and then we state our main result (Theorem~\ref{thm:main}) concerning the existence of three different asymptotic regimes, and some consequences (Corollary~\ref{3sol}). In section~\ref{sec:proof-sketch} we reduce the proof of our main result to the proof of four auxiliary propositions, where we concentrate the technical machinery of the paper, and we prove Corollary~\ref{3sol}. Finally, section~\ref{sec:proof-prop} is devoted to the proof of the propositions.

\setcounter{equation}{0}
\section{Statements}\label{sec:statements}

\subsection{Preliminary results}\label{Prelim}

We start by some basic properties appealing to rather classical techniques. 

\begin{thm}[Well-posedness]\label{thmbibl:wp}

Let $H$ be a Hilbert space, let $A$ be a self-adjoint nonnegative linear operator on $H$ with dense domain $D(A)$, let $\lambda$ be a real number, and let $f\in C^{0}([0,+\infty),H)$.

Then the following statements hold true.
\begin{enumerate}
   \renewcommand{\labelenumi}{(\arabic{enumi})}
  
	 \item \emph{(Global existence and uniqueness)} For every $(u_{0},u_{1})\in\dath$, problem (\ref{eqn:duffing})--(\ref{eqn:data}) admits a unique global solution
	  \begin{equation}
	  u\in C^{0}\left([0,+\infty),D(A)\right)\cap C^{1}\left([0,+\infty),H\right).
	  \nonumber
	  \end{equation}
  
  \item \emph{(Continuous dependence on initial data)} Let $\{(u_{0n},u_{1n})\}$ be any sequence with
  \begin{equation}
  (u_{0n},u_{1n})\to(u_{0},u_{1})
  \quad\quad
  \mbox{in }\dath,
  \nonumber
  \end{equation}
  and let $u_{n}(t)$ denote the solution to (\ref{eqn:duffing}) with initial data $u_{n}(0)=u_{0n}$ and $u'(0)=u_{1n}$.
  
  Then for every $T>0$ it turns out that
  \begin{equation}
  u_{n}(t)\to u(t)
  \quad\quad
  \mbox{uniformly in }C^{0}\left([0,T],D(A)\right),
  \nonumber
  \end{equation}
  \begin{equation}
  u_{n}'(t)\to u'(t)
  \quad\quad
  \mbox{uniformly in }C^{1}\left([0,T],H\right).
  \nonumber
  \end{equation}
   
  \item \emph{(Derivative of the energy)} The classical energy
  \begin{equation}
  E(t):=\frac{1}{2}|u'(t)|^{2}+\frac{1}{2}|Au(t)|^{2}-\frac{\lambda}{2}|A^{1/2}u(t)|^{2}+\frac{1}{4}|A^{1/2}u(t)|^{4}
  \label{defn:E}
  \end{equation}
  is of class $C^{1}$, and its time-derivative is given by
  \begin{equation}
  E'(t)=-|u'(t)|^{2}+\langle u'(t),f(t)\rangle
  \quad\quad
  \forall t\geq 0.
  \label{defn:E-classical}
  \end{equation} 
\end{enumerate}

\end{thm}

The proof of this result is quite standard. Introducing the vector $U(t): = (u(t), v(t))$, equation (\ref{eqn:duffing}) can be written on the product Hilbert space $\mathcal{H}:=D(A)\times H$ in the form 
$$ U' = LU + G(U) + (0, f(t)), $$  
where 
$$ LU := (v, -A^2u), 
\qquad 
\forall U \in D(L) := D(A^2)\times D(A), $$
and 
$$ G( U) := (0, -v+ \lambda Au  -  |A^{1/2}u|^2 Au).$$
 
The linear operator $L$ is skew-adjoint on $\cal{H}$, and the nonlinear term $G:D(A)\to H$ is Lipschitz continuous on bounded subsets. This is enough to deduce local existence through standard techniques (see for example~\cite{C-H}), as well as continuous dependence on initial data as soon as all solutions are global. In turn, global existence follows from the bounds on the classical energy (\ref{defn:E}) that can be deduced from (\ref{defn:E-classical}) through Gronwall's Lemma. Finally, (\ref{defn:E-classical}) is an immediate application of \cite[Lemma 11]{H-LN} combined with the observation that, since $ u\in C^{0}\left([0,+\infty),D(A)\right)\cap C^{1}\left([0,+\infty),H\right)$, we have 
$$F_1(t):= |A^{1/2}u(t)|^2\in C^{1}\left([0,+\infty),H\right) $$
with 
$$F'_1(t) = 2\langle Au(t),u'(t)\rangle.$$  

We leave the details to the reader.

\begin{rmk}\label{rmk:backward}
\begin{em}

The well-posedness result holds true also backward-in-time. In particular
\begin{itemize}
  \item if $f(t)$ is defined and continuous in the whole real line, then $u(t)$ is defined on the whole real line,
  \item if $f(t)$ is defined only for $t\geq 0$, but ``initial'' data are
  $$u(T_{0})=u_{0}\in D(A),
  \qquad\quad
  u'(T_{0})=u_{1}\in H$$
  for some $T_{0}\geq 0$, then the solution is again defined for every $t\geq 0$.
\end{itemize} 

\end{em}
\end{rmk}

In the sequel we restrict our analysis to the case where $\lambda$ lies between the first two eigenvalues of $A$. To be more precise, we introduce the following class of operators.

\begin{defn}[Operators with gap condition]\label{defn:gap}
\begin{em}

Let $H$ be a Hilbert space, and let $\lambda_{1}<\lambda_{2}$ be two positive real numbers. We say that an operator $A$ satisfies the $(\lambda_{1},\lambda_{2})$ gap condition, and we write $A\in\gaph$, if $A$ is a self-adjoint linear operator on $H$ with dense domain $D(A)$, and there exists $e_{1}\in H$, with $|e_{1}|=1$, such that
\begin{itemize}
  \item $Ae_{1}=\lambda_{1}e_{1}$,
  \item $|Ax|^{2}\geq\lambda_{2}^{2}|x|^{2}$ for every $x\in D(A)$ with $\langle x,e_{1}\rangle=0$.
\end{itemize}

\end{em}
\end{defn}

\begin{rmk}
\begin{em}
When $A\in\gaph$, the parameter $\lambda_{1}$ turns out to be the smallest eigenvalue of $A$, with corresponding unit eigenvector $e_{1}$, and the spectrum of $A$ does not intersect the interval $(\lambda_{1},\lambda_{2})$. A classical example of operator in $\gaph$ is any operator whose spectrum is a sequence $\lambda_{1}<\lambda_{2}<\ldots$ of positive real numbers, and whose first eigenvalue $\lambda_{1}$ is simple. The Laplacian with homogeneous Dirichlet boundary conditions  in any reasonable bounded domain fits in this framework. 
\end{em}
\end{rmk}

\begin{rmk}[Stationary solutions to the unforced equation]
\begin{em}

When $A\in\gaph$, the unforced equation 
\begin{equation}
u''+u'+A^{2}u-\lambda Au+|A^{1/2}u|^{2}Au=0.
\label{eqn:f=0}
\end{equation}
admits exactly three stationary solutions, namely those of the form $u(t)\equiv\sigma e_{1}$ with $\sigma\in\{-\sigma_{0},0,\sigma_{0}\}$, where 
\begin{equation}
\sigma_{0}:=\sqrt{\frac{\lambda-\lambda_{1}}{\lambda_{1}}}.
\label{defn:sigma-0}
\end{equation}

Indeed, stationary solutions to (\ref{defn:sigma-0}) are the solutions to the abstract elliptic equation 
$$ A^2u = \lambda Au  -  |A^{1/2}u|^2 Au = \mu Au $$  
with $ \mu :=\lambda  -  |A^{1/2}u|^2 $.  This means that either $u= 0$, or $Au$ is an eigenvector of $A$ with eigenvalue $\mu$. Since $\mu<\lambda <\lambda_2$, the only possibility is that $ \mu = \lambda_1$, in which case we deduce that $u=\sigma e_{1}$ for some $\sigma\in\{-\sigma_{0},0,\sigma_{0}\}$.

\end{em}
\end{rmk}

\begin{rmk}[Convergence to equilibria for the unforced equation]
\begin{em}

If $H$ is a finite dimensional space, then the classical invariance principle, applied with the classical energy (\ref{defn:E}) as Lyapunov function, proves that all solutions to the unforced equation (\ref{eqn:f=0}) converge to one of the three equilibria.

If the dimension of $H$ is infinite, the convergence to equilibria for solutions to the unforced equation is a corollary of our main result, but it does not follow immediately from the invariance principle. Indeed, in order to apply this principle one has to know a priori that trajectories are precompact, a property that does not seem so easy to obtain in infinite dimension without exploiting the full machinery introduced in the proof of our main result. 

\end{em}
\end{rmk}

\subsection{Main result and consequences} 

For our main result we consider the full equation (\ref{eqn:duffing}) with a \emph{small bounded} forcing term. We show that, if the external force is asymptotically small enough, then every solution lies for $t$ large in a neighborhood of one of the three stationary solutions to the unforced equation (\ref{eqn:f=0}). Moreover, any two solutions of (\ref{eqn:duffing}) that are eventually close to the same stationary solution of (\ref{eqn:f=0}) are actually asymptotic to each other as $t\to +\infty$. The statement is the following.

\begin{thm}[Asymptotic behavior for the equation with small external force]\label{thm:main}

Let $H$ be a Hilbert space, let $\lambda_{1}<\lambda<\lambda_{2}$ be three positive real numbers, let $A\in\gaph$, let $f:[0,+\infty)\to H$ be a bounded continuous function, and let $\sigma_{0}$ be defined by (\ref{defn:sigma-0}).

Then there exist two positive constants $\ep_{0}$ and $M_{0}$, depending only on the three parameters $\lambda$, $\lambda_{1}$, $\lambda_{2}$, for which the following statements hold true whenever
\begin{equation}
\limsup_{t\to+\infty}|f(t)|\leq\ep_{0}.
\label{hp:main}
\end{equation}

\begin{enumerate}
   \renewcommand{\labelenumi}{(\arabic{enumi})}
  
	 \item \emph{(Alternative)} For every solution $u(t)$ to (\ref{eqn:duffing}), there exists $\sigma\in\{-\sigma_{0},0,\sigma_{0}\}$ such that
	 \begin{equation}
	 \limsup_{t\to+\infty}\left(|u'(t)|+|A(u(t)-\sigma e_{1})|\strut\right)\leq M_{0}\limsup_{t\to+\infty}|f(t)|.
	 \label{th:main-alternative}
	 \end{equation}
	   
  \item \emph{(Asymptotic convergence)} If $u(t)$ and $v(t)$ are any two solutions to (\ref{eqn:duffing}) satisfying (\ref{th:main-alternative}) with the same $\sigma\in\{-\sigma_{0},0,\sigma_{0}\}$, then $u(t)$ and $v(t)$ are asymptotic to each other in the sense that
  \begin{equation}
  \lim_{t\to +\infty}\left(u(t)-v(t),u'(t)-v'(t)\strut\right)=(0,0)
  \quad\quad
  \mbox{in }\dath.
  \label{th:main-asymptotic}
  \end{equation}

	\item \emph{(Basin of attraction of stable solutions)} The set of initial data $(u_{0},u_{1})$ for which the solution to (\ref{eqn:duffing})--(\ref{eqn:data}) satisfies (\ref{th:main-alternative}) with a given $\sigma\in\{-\sigma_{0},\sigma_{0}\}$ is a nonempty open subset of $\dath$. 
	
	\item \emph{(Stable manifold of unstable solution)} The set of initial data $(u_{0},u_{1})$ for which the solution to (\ref{eqn:duffing})--(\ref{eqn:data}) satisfies (\ref{th:main-alternative}) with $\sigma=0$ is a nonempty closed subset of $\dath$.
	
\end{enumerate}

\end{thm}

Statement~(1) of Theorem~\ref{thm:main} implies in particular that all solutions to the unforced equation (\ref{eqn:f=0}) converge to one of the equilibria. Actually, the following more general result is a special case of (\ref{th:main-alternative}).

\begin{cor}\label{asymp}
Under the assumptions of Theorem~\ref{thm:main}, if in addition $f(t)\to 0$ as $t\to +\infty$, then every solution to the forced equation (\ref{eqn:duffing}) converges to one of the stationary solutions to the unforced equation (\ref{eqn:f=0}).
\end{cor}

Another important consequence of Theorem~\ref{thm:main} is the following. 

\begin{cor}\label{3sol}

Let $H$, $\lambda_{1}<\lambda<\lambda_{2}$, and $A\in\gaph$ be as in Theorem~\ref{thm:main}, let $\sigma_{0}$ be defined by (\ref{defn:sigma-0}), and let $f:\re\to H$ be a globally defined forcing term.

Then there exist two positive constants $\ep_{0}$ and $r_{0}$ with the following properties.

\begin{enumerate}
\renewcommand{\labelenumi}{(\arabic{enumi})}

\item \emph{(Existence of three special bounded solutions)} If $|f(t)|\leq\ep_{0}$ for every $t\in\re$, then for every $\sigma\in\{-\sigma_{0},0,\sigma_{0}\}$ there exists a unique solution $u_{\sigma}(t)$ to (\ref{eqn:duffing}) such that
\begin{equation}
|u_{\sigma}'(t)|^{2}+|A(u_{\sigma}(t)-\sigma e_{1})|^{2}\leq r_{0}
\qquad
\forall t\in\re.
\nonumber
\end{equation}

\item \emph{(Asymptotic convergence)} Every solution to (\ref{eqn:duffing}) is asymptotic in $D(A)\times H$, as $t\to +\infty$, to one of the three special solutions.

\item \emph{(Periodic case)} If $f(t)$ is $T$-periodic, then the three special solutions are $T$-periodic as well, and they are the unique  periodic solutions to the equation (in particular, subharmonic solutions do not exist).

\item \emph{(Almost periodic case)} If in addition $f(t)$ is almost periodic with values in $H$, then each of the three special solutions $u_{\sigma}(t)$ is such that the vector $ U_{\sigma}(t) : = (u_{\sigma}(t), u_{\sigma}'(t)) $ is almost periodic  with values in the energy space  $D(A)\times H$, with the module containment property. Moreover, the $u_{\sigma}(t)$ are the unique solutions of  the equation to be almost periodic with values in any Banach space $Z$ such that $H\subseteq Z $ with continuous imbedding.

\end{enumerate}

\end{cor}

\begin{rmk}\label{rmk:3sol}
\begin{em} 

When $f$ is just bounded and small, we do not claim that the three special solutions are the unique solutions that are bounded on the whole real line. Indeed, even when $f=0$, there are most probably infinitely many heteroclinic orbits which lie in the unstable manifold of the trivial equilibrium, and converge to $0$ as $t\to -\infty$, and to one of the two stable equilibria as $t\to +\infty$. These orbits do not fit in the framework of statement~(1) of Corollary~\ref{3sol} because the do not lie in a small neighborhood of any equilibrium.

\end{em}
\end{rmk}

\setcounter{equation}{0}
\section{Proof of main results}\label{sec:proof-sketch}

\subsection{Proof of Theorem~\ref{thm:main} -- Auxiliary results}\label{sec:strategy}

The proof of Theorem~\ref{thm:main} relies on four auxiliary results, whose proof is postponed to section~\ref{sec:proof-prop}. To begin with, we introduce the constant
\begin{equation}
\gamma_{0}:=\frac{1}{8}\min\left\{1,\lambda_{1}(\lambda-\lambda_{1}),\lambda_{2}(\lambda_{2}-\lambda)\strut\right\},
\label{defn:gamma0}
\end{equation}
and the linear operator $P$ on $H$ such that $Pe_{1}=e_{1}/6$, and $Px=x$ for every $x$ in the subspace of $H$ orthogonal to $e_{1}$.

For every solution $u(t)$ to (\ref{eqn:duffing}), we consider the classical energy $E(t)$ defined by (\ref{defn:E}), and the corrected energy
\begin{equation}
F(t):=E(t)+2\gamma_{0}\langle Pu(t),u'(t)\rangle+\gamma_{0}\langle Pu(t),u(t)\rangle.
\label{defn:F}
\end{equation}

In the first result we prove that the energy $F(t)$ of solutions to (\ref{eqn:duffing}) is bounded for $t$ large in terms of the norm of the forcing term. As a consequence, all solutions are bounded in $\dath$.

\begin{prop}[Ultimate bound on solutions]\label{prop:ultimate}

Let us consider equation (\ref{eqn:duffing}) under the assumptions of Theorem~\ref{thm:main}. Let $F(t)$ be the energy defined in (\ref{defn:F}).

Then there exists a positive constant $M_{1}$ such that
\begin{equation}
\limsup_{t\to+\infty}F(t)\leq M_{1}\limsup_{t\to+\infty}|f(t)|^{2},	
\label{th:ultimate}
\end{equation}
and there exist two positive constants $M_{2}$ and $M_{3}$ such that
\begin{equation}
\limsup_{t\to+\infty}\left(|u'(t)|^{2}+|Au(t)|^{2}\right)\leq M_{2}+M_{3}\limsup_{t\to+\infty}|f(t)|^{2}.	
\label{th:ultimate-u}
\end{equation}
\end{prop}

In the second result we deal with solutions which lie \emph{eventually}  in the instability region, namely with their \emph{first component} close to the origin. We show that in this case the \emph{whole solution} is eventually close to the origin, within a distance depending on the norm of the forcing term. 

\begin{prop}[Solutions in the unstable regime]\label{prop:unstable}

Let us consider equation (\ref{eqn:duffing}) under the assumptions of Theorem~\ref{thm:main}.

Then there exist two constants $\beta_{0}\in(0,\sigma_{0})$ and $M_{4}>0$ for which the following implication is true:
\begin{equation}
\boxa{\begin{array}{c}
\displaystyle\limsup_{t\to+\infty}|f(t)|\leq 1 \\
\noalign{\vspace{1ex}}
\displaystyle\limsup_{t\to+\infty}|\langle u(t),e_{1}\rangle|\leq\beta_{0}
\end{array}}
\Longrightarrow
\boxa{\limsup_{t\to+\infty}\left(|u'(t)|+|Au(t)|\strut\right)\leq M_{4}\limsup_{t\to+\infty}|f(t)|}
\label{th:unstable}
\end{equation}

\end{prop}

In the third result we consider solutions to (\ref{eqn:duffing}) which \emph{at a given time} are close to the stable stationary solution $\sigma_{0}e_{1}$ to (\ref{eqn:f=0}). We show that these solutions lie eventually in a neighborhood of the stationary solution itself, within a distance depending on the norm of the forcing term.

\begin{prop}[Solutions in the stable regime]\label{prop:stable}

Let us consider equation (\ref{eqn:duffing}) under the assumptions of Theorem~\ref{thm:main}. Let $F(t)$ be the energy defined in (\ref{defn:F}).

Then for every $\beta\in(0,\sigma_{0})$ there exist three constants $\eta>0$, $\ep_{1}>0$ and $M_{5}>0$ for which the following implication is true: 
\begin{equation}
\begin{array}{c}
\fbox{$\begin{array}{c}
\exists\, T_{0}\geq 0 \mbox{ such that } \\
\noalign{\vspace{1ex}}
\displaystyle\sup_{t\geq T_{0}}|f(t)|\leq\ep_{1},\ F(T_{0})<\eta,\ |\langle u'(T_{0}),e_{1}\rangle|<\eta,\ \langle u(T_{0}),e_{1}\rangle>\beta
\end{array}$}   \\
\noalign{\vspace{1ex}}
\Downarrow   \\
\noalign{\vspace{1ex}}
\boxa{\limsup_{t\to+\infty}\left(|u'(t)|+|A(u(t)-\sigma_{0}e_{1})|\strut\right)\leq M_{5}\limsup_{t\to+\infty}|f(t)|}  
\end{array}
\label{th:stable}
\end{equation}

\end{prop}

In the last result we show that any two solutions to the forced equation (\ref{eqn:duffing}) that are close enough to the \emph{same stationary solution} to the unforced equation (\ref{eqn:f=0}) are actually asymptotic to each other.

\begin{prop}[Close solutions are asymptotic to each other]\label{prop:asymptotic}

Let us consider equation (\ref{eqn:duffing}) under the assumptions of Theorem~\ref{thm:main}.

Then there exists $r_{0}>0$ with the following property: if $u(t)$ and $v(t)$ are two solutions to (\ref{eqn:duffing}), and there exists $\sigma\in\{-\sigma_{0},0,\sigma_{0}\}$ such that
\begin{equation}
\limsup_{t\to +\infty}\left(|u'(t)|+|A(u(t)-\sigma e_{1})|+|v'(t)|+|A(v(t)-\sigma e_{1})|\right)\leq r_{0},
\label{hp:asymptotic}
\end{equation}
then $u(t)$ and $v(t)$ are asymptotic to each other in the sense of (\ref{th:main-asymptotic}).

\end{prop}

\subsection{Proof of Theorem~\ref{thm:main} -- Conclusion}

In this section we prove Theorem~\ref{thm:main} by relying on the four propositions stated above.

To begin with, we consider the constants $\beta_{0}$ and $M_{4}$ of Proposition~\ref{prop:unstable}. Then we select $\beta:=\beta_{0}$ in Proposition~\ref{prop:stable}, and we obtain three more constants $\eta$, $\ep_{1}$, and $M_{5}$. We also consider the constant $M_{1}$ of Proposition~\ref{prop:ultimate}, and the constant $r_{0}$ of Proposition~\ref{prop:asymptotic}. We claim that the conclusions of Theorem~\ref{thm:main} hold true if we choose
$$M_{0}:=\max\{M_{4},M_{5}\},
\hspace{3em}
\ep_{0}:=\min\left\{1,\frac{\ep_{1}}{2},\frac{\eta}{2M_{1}},\frac{\eta}{2M_{0}},\frac{r_{0}}{2M_{0}},\frac{(\sigma_{0}-\beta_{0})\lambda_{1}}{2M_{0}}\right\}.$$

\paragraph{\textmd{\textit{Alternative}}}

Let us assume that (\ref{hp:main}) is satisfied, and let $u(t)$ be any solution to (\ref{eqn:duffing}). Let us set
$$L:=\limsup_{t\to +\infty}|\langle u(t),e_{1}\rangle|.$$

We observe that $L$ is finite because of (\ref{th:ultimate-u}), and we distinguish two cases.
\begin{itemize}
  \item Case $L\leq\beta_{0}$. Since $\ep_{0}\leq 1$, we can apply Proposition~\ref{prop:unstable}, from which we deduce that in this case $u(t)$ satisfies (\ref{th:main-alternative}) with $\sigma=0$.
  
  \item Case $L>\beta_{0}$. We can assume, without loss of generality, that
  \begin{equation}
  L=\limsup_{t\to +\infty}\,\langle u(t),e_{1}\rangle
  \label{hp:limsup}
  \end{equation}
  (without absolute value). Indeed, the other possibility can be treated in a symmetric way, or even reduced to this one by considering the function $-u(t)$, which is a solution to (\ref{eqn:duffing}) with external force $-f(t)$.
  
When (\ref{hp:limsup}) holds true, we claim that we are in the framework of Proposition~\ref{prop:stable} with $\beta=\beta_{0}$, namely there exists $T_{0}\geq 0$ for which the four inequalities in the upper part of (\ref{th:stable}) are satisfied. Indeed, since $\ep_{0}\leq\ep_{1}/2$, from assumption (\ref{hp:main}) we deduce that $|f(t)|\leq\ep_{1}$ when $t$ is large enough. Moreover, since $M_{1}\ep_{0}^{2}\leq M_{1}\ep_{0}\leq\eta/2$, from Proposition~\ref{prop:ultimate} we deduce that $F(t)<\eta$ for every $t$ large enough. Finally, from Lemma~\ref{lemma:limsup} applied to the function $\varphi(t):=\langle u(t),e_{1}\rangle$, we deduce that there exists a sequence $t_{n}\to+\infty$ such that 
  $$\langle u(t_{n}),e_{1}\rangle\to L>\beta_{0}
  \quad\mbox{and}\quad
  \langle u'(t_{n}),e_{1}\rangle\to 0.$$ 
  
  Therefore, all the four inequalities in the upper part of (\ref{th:stable}) are satisfied if we choose $T_{0}:=t_{n}$ with $n$ large enough. At this point, from (\ref{th:stable}) we conclude that in this case $u(t)$ satisfies (\ref{th:main-alternative}) with $\sigma=\sigma_{0}$.
 
\end{itemize} 

\paragraph{\textmd{\textit{Asymptotic convergence}}}

Since $M_{0}\ep_{0}\leq r_{0}/2$, any pair of solutions satisfying~(\ref{th:main-alternative}) with the same $\sigma\in\{-\sigma_{0},0,\sigma_{0}\}$ satisfies also (\ref{hp:asymptotic}) with the same $\sigma$. At this point, (\ref{th:main-asymptotic}) follows from Proposition~\ref{prop:asymptotic}.

\paragraph{\textmd{\textit{Basin of attraction of stable solutions}}}

Let us consider the case $\sigma=\sigma_{0}$ (but the argument is symmetric when $\sigma=-\sigma_{0}$). We claim that, under assumption (\ref{hp:main}) with our choice of $\ep_{0}$, the following characterization holds true: a solution to (\ref{eqn:duffing}) satisfies (\ref{th:main-alternative}) with $\sigma=\sigma_{0}$ if and only if there exists $T_{0}\geq 0$, possibly depending on the solution, for which the four inequalities in upper part of (\ref{th:stable}) hold true (as usual with $\beta$ equal to the value $\beta_{0}$ of Proposition~\ref{prop:unstable}).

Let us prove this characterization. The ``if part'' is exactly Proposition~\ref{prop:stable}. As for the ``only if part'', it is enough to show that (\ref{hp:main}) and (\ref{th:main-alternative}) imply that the four inequalities in the assumptions of (\ref{th:stable}) hold true when $T_{0}$ is large enough. The first one follows from (\ref{hp:main}) because $\ep_{0}\leq\ep_{1}/2$. The second one follows from Proposition~\ref{prop:ultimate} because $M_{1}\ep_{0}^{2}\leq\eta/2$, as explained before. The third one follows from  (\ref{th:main-alternative}) because  $M_{0}\ep_{0}\leq\eta/2$. Finally, from (\ref{th:main-alternative}) and our definition of $\ep_{0}$ we obtain that
$$\limsup_{t\to+\infty}|\langle u(t),e_{1}\rangle-\sigma_{0}|\leq\limsup_{t\to+\infty}\frac{1}{\lambda_{1}}|A(u(t)-\sigma_{0}e_{1})|\leq\frac{M_{0}\ep_{0}}{\lambda_{1}}\leq\frac{\sigma_{0}-\beta_{0}}{2},$$
and hence $|\langle u(t),e_{1}\rangle-\sigma_{0}|<\sigma_{0}-\beta_{0}$ when $t$ is large enough. This implies that
$$\langle u(t),e_{1}\rangle=\sigma_{0}+\left(\langle u(t),e_{1}\rangle-\sigma_{0}\strut\right)>\sigma_{0}-(\sigma_{0}-\beta_{0})=\beta_{0}$$
when $t$ is large enough, from which the last required inequality follows.

Given the characterization, we can prove our conclusions. Indeed, due to the continuous dependence on initial data, the set of initial data $(u_{0},u_{1})$ originating a solution $u(t)$ for which there required $T_{0}$ exists is an open set. In order to prove that it is nonempty, we choose $T_{0}\geq 0$ such that $|f(t)|\leq 2\ep_{0}\leq\ep_{1}$ for every $t\geq T_{0}$, and we consider the solution $u(t)$ to (\ref{eqn:duffing}) with ``initial'' data
$$u(T_{0})=\sigma_{0}e_{1},
\qquad\quad
u'(T_{0})=0.$$

This solution is defined on the whole half-line $t\geq 0$ because of Remark~\ref{rmk:backward}, and it fits in the assumptions of Proposition~\ref{prop:stable}. Indeed, among the four inequalities in the upper part of (\ref{th:stable}), the only nontrivial is that $F(T_{0})<\eta$, and this is true because 
$$F(T_{0})=\frac{\lambda_{1}^{2}}{2}\sigma_{0}^{2}-\frac{\lambda\lambda_{1}}{2}\sigma_{0}^{2}+\frac{\lambda_{1}^{2}}{4}\sigma_{0}^{4}+\frac{\gamma_{0}}{6}\sigma_{0}^{2}=\left(-\frac{\lambda_{1}(\lambda-\lambda_{1})}{4}+\frac{\gamma_{0}}{6}\right)\frac{\lambda-\lambda_{1}}{\lambda_{1}}$$
is negative when $\gamma_{0}$ in chosen according to (\ref{defn:gamma0}).

\paragraph{\textmd{\textit{Stable manifold of unstable solution}}}

Due to the alternative of statement~(1), the set of initial data originating a solution satisfying (\ref{th:main-alternative}) with $\sigma=0$ is the complement of the basins of attraction of the stable solutions. Since the two basins of attraction are open, this set is necessarily closed, and nonempty because the phase space $\dath$ is connected and cannot be represented as the union of two disjoint open sets.\qed 

\subsection{Proof of Corollary~\ref{3sol}} 

\paragraph{\textmd{\textit{Existence of three special solutions}}}

The argument is a quite standard perturbation method, and parallels the proof of the scalar case (see~\cite{haraux:duffing}). The basic idea is to linearize around the three fixed points. Let us introduce for convenience $V: = D(A) $  endowed with the natural norm. We deal with  an equation of the form
\begin{equation}\label{perturb}
w''+w'+Lw=f(t)+g(w),
\end{equation}
where $L\in {\cal {L}}(V, V') $ is the linear operator defined by
\begin{equation}
Lw:=\left\{
\begin{array}{ll}
 A^{2}w-\lambda Aw & \mbox{around the trivial equilibrium},   \\[1ex]
 A^{2}w-\lambda_1 Aw + 2\lambda_1(\lambda - \lambda_1 )\langle e_1, w\rangle e_1 & \mbox{around }\pm\sigma_{0}e_{1},   
 \end{array}
\right.
\nonumber
\end{equation}
and $g\in C(V, H) $ is a nonlinear locally Lipschitz continuous map with 
$$ |g(u) - g(v)| \le G (\|u\|_V + \|v\|_V)\|u-v\|_V$$  
for some constant $G>0$ whenever $\|u\|_V + \|v\|_V$ is small enough. We introduce the Banach spaces 
$$ X = C_b (\re, V) \cap C^1_b (\re, H)\cap C^2_b (\re, V') ; \quad  Y = C_b (\re, H).$$ 

We observe that when 
$$ Lw = A^{2}w -\lambda_1 Aw + 2\lambda_1(\lambda - \lambda_1 )\langle e_1, w\rangle e_1 , $$  
on ${e_1}^\perp$  we have 
$ L = A^{2}-\lambda_1 A $,  and on $\re e_1$ we have $L = 2\lambda_1(\lambda - \lambda_1 )I$.  Clearly $L$ is self-adjoint and coercive with domain $D(L) = D(A^2)$. 
On the other hand when 
$$ L = A^{2}-\lambda A , $$ 
since $\lambda_1< \lambda< \lambda_2$ there is  just one negative eigenvalue and we can use exponential dichotomy. Introducing 
$$ \forall w \in X, \quad \Lambda w= w''+w'+Lw $$ 
it is now standard that for all $h\in Y$, there is exactly one function $w\in X$ such that $\Lambda w = h $  and moreover $ K: = {\Lambda}^{-1}\in {\cal {L}}(Y, X). $  Now equation \eqref{perturb} can be written in the form 
$$ w = K(f + g(w))$$  
which suggests to study the dependence of $w\in X$ in terms of $v$ in the equation \begin{equation}
w = K(f + g(v))=: F(f, v)
\nonumber
\end{equation} 

It is readily verified that for $f$ fixed, $F$ is a contraction with Lipschitz norm less than $1/2$ when $w$ lies in a sufficiently small closed ball  $ B(0, r)$ centered at $0$ in $X$. Moreover, if $f$ is small enough in $Y$, the inequality 
$$ \|F(f, v)\|_X \le 1/2 \|v\|_X + \| K\|_{{\cal {L}}(Y, X)} \|f\|_Y $$ 
shows that the map $ v\longrightarrow F(f, v)$ leaves the ball $ B(0, r)$ invariant. The standard fixed point theorem concludes the proof, including the uniqueness statement and the Lipschitz dependence of the solution in $X$ in terms of $f\in Y$. We leave the remaining details to the reader.

\paragraph{\textmd{\textit{Asymptotic convergence}}}

This follows from statement~(2) of Theorem~\ref{thm:main}, provided that $\ep_{0}$ and $r_{0}$ are small enough.

\paragraph{\textmd{\textit{Periodic case}}}

Uniqueness follows from statement~(2) of Theorem~\ref{thm:main} (two periodic functions that are asymptotic to each other coincide necessarily). The existence of three periodic solutions follows from the fact that if $f$ is $T$-periodic and $u$ is the special bounded  solution, then $u(t+T)$ is a solution for the same $f$ located in the same small ball of $X$, therefore it must coincide with $u$.

\paragraph{\textmd{\textit{Almost periodic case}}}

When $f$ is almost periodic, both  almost periodicity of the three special solutions, and the module containment property, follow in a standard way from the fact that  the three special solutions depend on $f$ in a Lipschitz continuous way. 
In alternative, one could rely on the second almost periodicity criterion (double sequence characterization) of Bochner (see~\cite[Theorem~1]{Bochner}), but this requires to know that bounded solutions have precompact range in the energy space, which does not seem easy to obtain.

As for uniqueness, it follows again from statement~(2) of Theorem~\ref{thm:main}. Indeed, two almost periodic functions that are asymptotic to each other coincide necessarily, and it is sufficient for that to assume almost periodicity of the component $u$ instead of the vector $(u, u')$, with values in an arbitrarily large Banach space in which either $V$ or even $H$ is continuously imbedded.\qed

\setcounter{equation}{0}
\section{Proof of auxiliary results}\label{sec:proof-prop}

\subsection{Some useful ultimate bounds }
In this subsection we state and prove some simple properties that are going to be useful in the proof of our main result.

\begin{lemma}\label{lemma:limsup}

For every bounded function $\varphi:[0,+\infty)\to\re$ of class $C^{1}$ there exists a sequence $t_{n}\to +\infty$ of nonnegative real numbers such that
$$\lim_{n\to+\infty}\varphi'(t_{n})=0
\quad\quad\mbox{and}\quad\quad
\lim_{n\to+\infty}\varphi(t_{n})=\limsup_{t\to+\infty}\varphi(t).$$

\end{lemma}  

\begin{proof}  Let us set
$$\ell:=\liminf_{t\to +\infty}\varphi(t),
\qquad\qquad
L:=\limsup_{t\to +\infty}\varphi(t).$$

If $\ell=L$, then $\varphi(t)$ has a finite limit as $t\to +\infty$. Due to the mean value theorem, for every positive integer $n$ there exists $t_{n}\in(n,n+1)$ such that
$$\varphi(n+1)-\varphi(n)=\varphi'(t_{n}).$$

Since the left-hand side tends to 0, the sequence $t_{n}$ settles the matter in this case.

If $\ell<L$, then there exist two sequences $x_{n}\to+\infty$ and $y_{n}\to+\infty$ such that
\begin{equation}
x_{n}<y_{n}<x_{n+1},
\qquad
\varphi(x_{n})\leq\frac{L+\ell}{2},
\qquad
\varphi(y_{n})\geq L-\frac{1}{n}
\label{eqn:xn-yn}
\end{equation}
for every positive integer $n$. Let $t_{n}$ denote one of the maximum points (there might be infinitely many of them) of $\varphi(t)$ in $[x_{n},x_{n+1}]$ . Since $y_{n}\in [x_{n},x_{n+1}]$, clearly $\varphi(t_{n})\geq\varphi(y_{n})$, and from the last inequality in (\ref{eqn:xn-yn}) it follows that $\varphi(t_{n})\to L$, as required. Keeping the first condition in (\ref{eqn:xn-yn}) into account, we know that $t_{n}$ is not one of the endpoints of the interval $[x_{n},x_{n+1}]$ when $n$ is large enough, therefore $\varphi'(t_{n})=0$ when $n$ is large enough.\end{proof}

\begin{lemma}\label{lemma:ODE-limsup}

Let $m$ be a positive real number, let $\psi:[0,+\infty)\to\re$ be a continuous function, and let $y\in C^2([0,+\infty), \re) $ be a solution to
\begin{equation}
y''(t)+y'(t)-my(t)=\psi(t).
\label{hp:ODE-limsup}
\end{equation}

Let us assume that both $\psi(t)$ and $y(t)$ are bounded.

Then it turns out that
\begin{equation}
\limsup_{t\to+\infty}|y(t)|\leq\frac{1}{m}\limsup_{t\to+\infty}|\psi(t)|,
\label{th:lemma-ODE-u}
\end{equation}
\begin{equation}
\limsup_{t\to+\infty}|y'(t)|\leq 2\limsup_{t\to+\infty}|\psi(t)|.
\label{th:lemma-ODE-u'}
\end{equation}

\end{lemma}

\begin{proof}

First, interpreting (\ref{hp:ODE-limsup}) as a first order equation with unknown $y'(t)$, we obtain 
\begin{equation}
y'(t)=y'(0)e^{-t}+e^{-t}\int_{0}^{t}(\psi(s)+my(s))e^{s}\,ds.
\nonumber
\end{equation}
It  follows clearly  that $y'$   is bounded  with 
\begin{equation}
\limsup_{t\to+\infty}|y'(t)|\leq \limsup_{t\to+\infty}(|\psi(t)| +m|y(t)|).
\nonumber
\end{equation}

In particular \eqref {th:lemma-ODE-u'} will be an immediate consequence of \eqref {th:lemma-ODE-u}. By changing all functions to their opposites, to prove \eqref {th:lemma-ODE-u} it is clearly enough  to establish 
\begin{equation}
\limsup_{t\to+\infty}my(t)\leq\limsup_{t\to+\infty}(-\psi(t)).
\nonumber
\end{equation} 

This inequality can be viewed as a topological variant on the half-line of the classical maximum principle in a bounded interval relying on the simple idea that a function achieving its maximum at an interior point has a zero derivative and a non-positive second derivative there.  As a consequence of Lemma 6.2 from \cite{haraux:duffing} we know the existence of a sequence of reals $t_n\geq 0$  such that $t_n\rightarrow +\infty$ and  
$$ \limsup_{n\rightarrow +\infty} y'' (t_n)\le 0,
\qquad
\lim_{n\rightarrow
\infty}y(t_n)= \limsup_{t\to+\infty}y(t)$$ 

Since by the equation $y''$ is bounded, it is immediate to see, for instance reasoning by contradiction, that 
$$\lim_{n\rightarrow
\infty}y'(t_n)= 0.$$ 

Then the result becomes an immediate consequence of the equation.
\end{proof}

\begin{lemma}\label{lemma:PDE-limsup}

Let $X$ be a Hilbert space, and let $B$ be a self-adjoint linear operator on $X$ with dense domain $D(B)$. Let us assume that there exists a constant $m>0$ such that
\begin{equation}
\langle Bx,x\rangle\geq m|x|^{2}
\quad\quad
\forall x\in D(B).
\nonumber
\end{equation}

Let $\psi:[0,+\infty)\to X$ be a bounded continuous function, and let $y:[0,+\infty)\to X$ be a solution to
\begin{equation}
y''(t)+y'(t)+By(t)=\psi(t).
\nonumber
\end{equation}

Then it turns out that
\begin{equation}
\limsup_{t\to+\infty}\left(|y'(t)|^{2}+|B^{1/2}y(t)|^{2}\right)\leq 9\max\left\{1,\frac{1}{m}\right\}\cdot\limsup_{t\to+\infty}|\psi(t)|^2.
\nonumber
\end{equation}

\end{lemma}

\begin{proof} This is an immediate consequence of \cite[formula~(3.6)]{Aloui-haraux}. 
\end{proof}

\subsection{General notation and energies}

Let $H_{+}$ denote the subspace of $H$ orthogonal to $e_{1}$. Let us write $u(t)$ in the form
$$u(t)=u_{-}(t)e_{1}+u_{+}(t),$$
where
$u_{-}(t)=\langle u(t),e_{1}\rangle$ is the component of $u(t)$ with respect to $e_{1}$, and $u_{+}(t)$ is the orthogonal projection of $u(t)$ in $H_{+}$. Similarly, let us write the forcing term in the form $f(t)=f_{-}(t)e_{1}+f_{+}(t)$.

In this setting, equation (\ref{eqn:duffing}) is equivalent to the system (for the sake of shortness, we do not write the explicit dependence on $t$ in the left-hand sides)
\begin{equation}
u_{-}''+u_{-}'-\lambda_{1}(\lambda-\lambda_{1})u_{-}+\lambda_{1}^{2}(u_{-})^{3}+\lambda_{1}\left|A^{1/2}u_{+}\right|^{2}u_{-}=f_{-}(t),
\label{eqn:u-}
\end{equation}
\begin{equation}
u_{+}''+u_{+}'+A^{2}u_{+}-\lambda Au_{+}+\left|A^{1/2}u_{+}\right|^{2}Au_{+}+\lambda_{1}(u_{-})^{2}Au_{+}=f_{+}(t),
\label{eqn:u+}
\end{equation}
where the first one is a scalar equation, and the second one is an equation in $H_{+}$.

The energy $E(t)$ defined in (\ref{defn:E-classical}) can be decomposed as
\begin{equation}
E(t)=E_{-}(t)+E_{+}(t)+I(t),
\nonumber
\end{equation}
where
\begin{equation}
E_{-}(t):=\frac{1}{2}|u_{-}'(t)|^{2}-\frac{\lambda_{1}(\lambda-\lambda_{1})}{2}|u_{-}(t)|^{2}+\frac{\lambda_{1}^{2}}{4}|u_{-}(t)|^{4},
\label{defn:E-}
\end{equation}
\begin{equation}
E_{+}(t):=\frac{1}{2}|u_{+}'(t)|^{2}+\frac{1}{2}|Au_{+}(t)|^{2}-\frac{\lambda}{2}|A^{1/2}u_{+}(t)|^{2}+\frac{1}{4}|A^{1/2}u_{+}(t)|^{4},
\label{defn:E+}
\end{equation}
\begin{equation}
I(t):=\frac{\lambda_{1}}{2}|u_{-}(t)|^{2}\cdot|A^{1/2}u_{+}(t)|^{2}.
\label{defn:I}
\end{equation}

We can interpret $E_{-}(t)$ and $E_{+}(t)$ as the contributions of $u_{-}(t)$ and $u_{+}(t)$ to the total energy $E(t)$, with $I(t)$ representing some sort of interaction term between the components.

Analogously, the energy $F(t)$ defined in (\ref{defn:F}) can be decomposed as
\begin{equation}
F(t)=F_{-}(t)+F_{+}(t)+I(t),
\label{decomp:F}
\end{equation}
where $I(t)$ is the same as above, and
\begin{equation}
F_{-}(t):=E_{-}(t)+\frac{\gamma_{0}}{3}u_{-}(t)\cdot u_{-}'(t)+\frac{\gamma_{0}}{6}|u_{-}(t)|^{2},
\label{defn:F-}
\end{equation}
\begin{equation}
F_{+}(t):=E_{+}(t)+2\gamma_{0}\langle u_{+}(t),u_{+}'(t)\rangle+\gamma_{0}|u_{+}(t)|^{2}.
\label{defn:F+}
\end{equation}

\paragraph{\textmd{\textit{Estimates on the energy of the first component}}}

We show that, for every solution $u(t)$ to (\ref{eqn:duffing}), and every $t\geq 0$, it turns out that
\begin{equation}
F_{-}(t)\geq\frac{\lambda_{1}^{2}}{4}|u_{-}(t)|^{4}-\frac{\lambda_{1}(\lambda-\lambda_{1})}{2}|u_{-}(t)|^{2},
\label{est:F-equiv}
\end{equation}
and
\begin{equation}
F_{-}'(t)\leq-\gamma_{0}F_{-}(t)+2|f_{-}(t)|^{2}-\lambda_{1}|A^{1/2}u_{+}(t)|^{2}\cdot u_{-}(t)\cdot u_{-}'(t).
\label{est:F-'}
\end{equation}

To begin with, from the definition of $\gamma_{0}$ we find that
$$\frac{\gamma_{0}}{3}\left|u_{-}(t)\cdot u_{-}'(t)\strut\right|\leq\frac{\gamma_{0}}{6}|u_{-}'(t)|^{2}+\frac{\gamma_{0}}{6}|u_{-}(t)|^{2}\leq\frac{1}{2}|u_{-}'(t)|^{2}+\frac{\gamma_{0}}{6}|u_{-}(t)|^{2}.$$
Plugging this inequality into (\ref{defn:F-}) we obtain (\ref{est:F-equiv}). In order to prove (\ref{est:F-'}), with some algebra we write the time-derivative of $F_{-}(t)$ in the form
\begin{eqnarray}
F_{-}'(t) & = & -\gamma_{0}F_{-}(t)-\left(1-\frac{5\gamma_{0}}{6}\right)|u_{-}'(t)|^{2}-\frac{\gamma_{0}}{6}\lambda_{1}(\lambda-\lambda_{1})|u_{-}(t)|^{2}-\frac{\gamma_{0}}{12}\lambda_{1}^{2}|u_{-}(t)|^{4} 
\nonumber  \\
 &  & \mbox{}-\lambda_{1}|A^{1/2}u_{+}(t)|^{2}\cdot u_{-}(t)\cdot u_{-}'(t)-\frac{2\gamma_{0}}{3}I(t)+\frac{\gamma_{0}^{2}}{6}|u_{-}(t)|^{2}  
 \nonumber  \\
 &  &   \mbox{}+\frac{\gamma_{0}^{2}}{3}u_{-}(t)\cdot u_{-}'(t)+\frac{\gamma_{0}}{3}u_{-}(t)\cdot f_{-}(t)+u_{-}'(t)\cdot f_{-}(t).
 \label{est:F-'-step1}
\end{eqnarray}

The terms in the last line can be estimated as follows
\begin{eqnarray*}
 & \displaystyle\frac{\gamma_{0}^{2}}{3}u_{-}(t)\cdot u_{-}'(t)\leq\frac{\gamma_{0}^{2}}{6}|u_{-}(t)|^{2}+\frac{\gamma_{0}^{2}}{6}|u_{-}'(t)|^{2}, & \\[1ex]
 & \displaystyle\frac{\gamma_{0}}{3}u_{-}(t)\cdot f_{-}(t)\leq\frac{\gamma_{0}^{2}}{36}|u_{-}(t)|^{2}+|f_{-}(t)|^{2},  & \\[1ex]
 & \displaystyle u_{-}'(t)\cdot f_{-}(t)\leq\frac{1}{4}|u_{-}'(t)|^{2}+|f_{-}(t)|^{2}. &
\end{eqnarray*}

Plugging all these estimates into (\ref{est:F-'-step1}) if follows that
\begin{eqnarray}
F_{-}'(t) & \leq & -\gamma_{0}F_{-}(t)-\left(\frac{3}{4}-\frac{5\gamma_{0}}{6}-\frac{\gamma_{0}^{2}}{6}\right)|u_{-}'(t)|^{2}-\frac{\gamma_{0}}{6}\lambda_{1}(\lambda-\lambda_{1})|u_{-}(t)|^{2} 
\nonumber   \\
\noalign{\vspace{1ex}}
 &  & \mbox{}+\frac{13\gamma_{0}^{2}}{36}|u_{-}(t)|^{2}+2|f_{-}(t)|^{2}-\lambda_{1}|A^{1/2}u_{+}(t)|^{2}\cdot u_{-}(t)\cdot u_{-}'(t). 
 \label{est:F-'-interm}
\end{eqnarray}

Finally, our choice of $\gamma_{0}$ guarantees that
$$\frac{3}{4}-\frac{5\gamma_{0}}{6}-\frac{\gamma_{0}^{2}}{6}\geq 0,
\hspace{4em}
-\frac{\gamma_{0}}{6}\lambda_{1}(\lambda-\lambda_{1})+\frac{13\gamma_{0}^{2}}{36}\leq 0,$$
and therefore (\ref{est:F-'-interm}) implies (\ref{est:F-'}).

\paragraph{\textmd{\textit{Estimates on the energy of high frequencies}}}

We show that, for every solution $u_{+}(t)$ to (\ref{eqn:u+}), and every $t\geq 0$, it turns out that
\begin{equation}
F_{+}(t)\geq\min\left\{\frac{1}{4},\frac{\lambda_{2}-\lambda}{2\lambda_{2}}\right\}\left(|u_{+}'(t)|^{2}+|Au_{+}(t)|^{2}\strut\right),
\label{est:F+equiv}
\end{equation}
and
\begin{eqnarray}
F_{+}'(t) & \leq  &  -\gamma_{0}F_{+}(t)-\frac{1}{4}|u_{+}'(t)|^{2}-\frac{\lambda_{2}-\lambda}{\lambda_{2}}\gamma_{0}|Au_{+}(t)|^{2} +2|f_{+}(t)|^{2}
\nonumber \\
\noalign{\vspace{1ex}}
 &   &  \mbox{}-\lambda_{1}\langle Au_{+}(t),u_{+}'(t)\rangle\cdot|u_{-}(t)|^{2}-4\gamma_{0}I(t).
\label{est:F+'}
\end{eqnarray}

In order to prove (\ref{est:F+equiv}), we observe that
\begin{equation}
|2\gamma_{0}\langle u_{+}(t),u_{+}'(t)\rangle|\leq\gamma_{0}|u_{+}'(t)|^{2}+\gamma_{0}|u_{+}(t)|^{2},
\label{est:mixt-F+}
\end{equation}
and we exploit the coerciveness of $A$ in $H_{+}$ in order to deduce that
\begin{equation}
|Au_{+}(t)|^{2}-\lambda|A^{1/2}u_{+}(t)|^{2}\geq\frac{\lambda_{2}-\lambda}{\lambda_{2}}|Au_{+}(t)|^{2}.
\label{est:F+H+}
\end{equation}

Plugging these estimates into (\ref{defn:E+}) and (\ref{defn:F+}) we conclude that
$$F_{+}(t)\geq\left(\frac{1}{2}-\gamma_{0}\right)|u_{+}'(t)|^{2}+\frac{\lambda_{2}-\lambda}{2\lambda_{2}}|Au_{+}(t)|^{2},$$
which implies (\ref{est:F+equiv}) because $\gamma_{0}\leq 1/4$.

In order to prove (\ref{est:F+'}), with some algebra we write the time-derivative of $F_{+}(t)$ in the form
\begin{eqnarray}
F_{+}'(t) & = & -\gamma_{0}F_{+}(t)-\left(1-\frac{5\gamma_{0}}{2}\right)|u_{+}'(t)|^{2}-\frac{3\gamma_{0}}{2}\left(|Au_{+}(t)|^{2}-\lambda|A^{1/2}u_{+}(t)|^{2}\strut\right) 
\nonumber  \\
\noalign{\vspace{1ex}}
 &  & \mbox{}-\frac{7\gamma_{0}}{4}|A^{1/2}u_{+}(t)|^{4}-\lambda_{1}\langle Au_{+}(t),u_{+}'(t)\rangle\cdot|u_{-}(t)|^{2}-4\gamma_{0}I(t)+\gamma_{0}^{2}|u_{+}(t)|^{2} 
\nonumber   \\
\noalign{\vspace{1ex}}
 &  &   \mbox{}+2\gamma_{0}^{2}\langle u_{+}(t),u_{+}'(t)\rangle+2\gamma_{0}\langle u_{+}(t),f_{+}(t)\rangle+\langle u_{+}'(t),f_{+}(t)\rangle.
 \label{est:F+'-step1}
\end{eqnarray}

The first term in the last line can be estimated as in (\ref{est:mixt-F+}). The second and third term can be estimated as follows
$$2\gamma_{0}\langle u_{+}(t),f_{+}(t)\rangle\leq\gamma_{0}^{2}|u_{+}(t)|^{2}+|f_{+}(t)|^{2},$$
$$\langle u_{+}'(t),f_{+}(t)\rangle\leq\frac{1}{4}|u_{+}'(t)|^{2}+|f_{+}(t)|^{2}.$$

Finally, from the coercivity of $A$ in $H_{+}$ we obtain both (\ref{est:F+H+}) and
$$|u_{+}(t)|^{2}\leq\frac{1}{\lambda_{2}^{2}}|Au_{+}(t)|^{2}.$$

Plugging all these estimates into (\ref{est:F+'-step1}) we deduce that
\begin{eqnarray}
F_{+}'(t) & \leq & -\gamma_{0}F_{+}(t)-\left(\frac{3}{4}-\frac{5\gamma_{0}}{2}-\gamma_{0}^{2}\right)|u_{+}'(t)|^{2}-\left(\frac{3}{2}\frac{\lambda_{2}-\lambda}{\lambda_{2}}-\frac{3\gamma_{0}}{\lambda_{2}^{2}}\right)\gamma_{0}|Au_{+}(t)|^{2} 
\nonumber   \\
\noalign{\vspace{1ex}}
 &  & \mbox{}+2|f_{+}(t)|^{2}-\lambda_{1}\langle Au_{+}(t),u_{+}'(t)\rangle\cdot|u_{-}(t)|^{2}-4\gamma_{0}I(t). 
 \label{est:F+'-interm}
\end{eqnarray}

Finally, our choice (\ref{defn:gamma0}) of $\gamma_{0}$ guarantees that
$$\frac{3}{4}-\frac{5\gamma_{0}}{2}-\gamma_{0}^{2}\geq\frac{1}{4},
\hspace{4em}
\frac{3}{2}\frac{\lambda_{2}-\lambda}{\lambda_{2}}-\frac{3\gamma_{0}}{\lambda_{2}^{2}}\geq\frac{\lambda_{2}-\lambda}{\lambda_{2}},$$
and therefore (\ref{est:F+'-interm}) implies (\ref{est:F+'}).

\subsection{Proof of Proposition~\ref{prop:ultimate}}

From (\ref{decomp:F}) it follows that
$$F'(t)=F_{-}'(t)+F_{+}'(t)+I'(t)
\quad\quad
\forall t\geq 0.$$

Now we estimate the first two terms as in (\ref{est:F-'}) and (\ref{est:F+'}), and we observe that
\begin{equation}
I'(t)=\lambda_{1}\langle Au_{+}(t),u_{+}'(t)\rangle\cdot|u_{-}(t)|^{2}+\lambda_{1}|A^{1/2}u_{+}(t)|^{2}\cdot u_{-}(t)\cdot u_{-}'(t).
\label{eqn:I'}
\end{equation}

We deduce that
$$F'(t)\leq -\gamma_{0}F(t)+2|f(t)|^{2}
\quad\quad
\forall t\geq 0.$$

Integrating this differential inequality we obtain (\ref{th:ultimate}) with $M_{1}:=2/\gamma_{0}$.

In order to prove (\ref{th:ultimate-u}), we write $F(t)$ in the form (\ref{defn:F}), and we observe that
$$\frac{1}{4}x^{4}-\frac{\lambda}{2}x^{2}\geq-\frac{\lambda^{2}}{4}
\quad\quad
\forall x\in\re,$$
and 
$$\left|2\gamma_{0}\langle Pu(t),u'(t)\rangle\strut\right|\leq\gamma_{0}|P^{1/2}u'(t)|^{2}+\gamma_{0}|P^{1/2}u(t)|^{2}\leq\gamma_{0}|u'(t)|^{2}+\gamma_{0}|P^{1/2}u(t)|^{2}.$$

Plugging the first estimate into (\ref{defn:E}), and the second one into (\ref{defn:F}), we obtain that
$$F(t)\geq\left(\frac{1}{2}-\gamma_{0}\right)|u'(t)|^{2}+\frac{1}{2}|Au(t)|^{2}-\frac{\lambda^{2}}{4}
\quad\quad
\forall t\geq 0.$$

Since $\gamma_{0}\leq 1/4$, this proves that
$$|u'(t)|^{2}+|Au(t)|^{2}\leq 4F(t)+\lambda^{2}
\quad\quad
\forall t\geq 0.$$

At this point, (\ref{th:ultimate-u}) follows from (\ref{th:ultimate}).\qed

\subsection{Proof of Proposition~\ref{prop:unstable}}

Let us choose $\beta_{0}>0$ small enough so that
\begin{equation}
16\beta_{0}^{4}\lambda_{1}^{2}\leq\frac{\lambda_{2}-\lambda}{\lambda_{2}}\gamma_{0},
\hspace{3em}
\beta_{0}^{2}\leq\frac{\lambda-\lambda_{1}}{2\lambda_{1}}.
\label{defn:beta0}
\end{equation}

Let us write as usual $u(t)=u_{-}(t)e_{1}+u_{+}(t)$, and $f(t)=f_{-}(t)e_{1}+f_{+}(t)$. The constants $c_{1}$, \ldots, $c_{6}$ in the sequel depend only on the three parameters $\lambda$, $\lambda_{1}$, $\lambda_{2}$.

\paragraph{\textmd{\textit{Estimate on high frequencies}}}

We show that
\begin{equation}
\limsup_{t\to +\infty}\left(|u_{+}'(t)|^{2}+|Au_{+}(t)|^{2}\right)\leq c_{1}\limsup_{t\to +\infty}|f_{+}(t)|^{2}.
\label{th:unstable+}
\end{equation}

To this end, let us consider the energy $F_{+}(t)$ defined in (\ref{defn:F+}). Since we assumed that the limsup of $|u_{-}(t)|$ is less than or equal to $\beta_{0}$, there exists $t_{0}\geq 0$ such that $|u_{-}(t)|\leq 2\beta_{0}$ for every $t\geq t_{0}$. Keeping into account the first inequality in (\ref{defn:beta0}), it follows that
\begin{eqnarray*}
-\lambda_{1}|u_{-}(t)|^{2}\cdot\langle Au_{+}(t),u_{+}'(t)\rangle & \leq & \frac{1}{4}|u_{+}'(t)|^{2}+\lambda_{1}^{2}|u_{-}(t)|^{4}\cdot|Au_{+}(t)|^{2} \\
\noalign{\vspace{1ex}}
 & \leq & \frac{1}{4}|u_{+}'(t)|^{2}+16\beta_{0}^{4}\lambda_{1}^{2}\cdot|Au_{+}(t)|^{2} \\
\noalign{\vspace{1ex}}
 & \leq & \frac{1}{4}|u_{+}'(t)|^{2}+\frac{\lambda_{2}-\lambda}{\lambda_{2}}\gamma_{0}|Au_{+}(t)|^{2}
\end{eqnarray*}
for every $t\geq t_{0}$. Plugging this estimate into (\ref{est:F+'}), and recalling that $I(t)$ is nonnegative, we deduce that
$$F_{+}'(t)\leq-\gamma_{0}F_{+}(t)+2|f_{+}(t)|^{2}
\quad\quad
\forall t\geq t_{0}.$$

Integrating this differential inequality, we conclude that
$$\limsup_{t\to +\infty}F_{+}(t)\leq\frac{2}{\gamma_{0}}\limsup_{t\to +\infty}|f_{+}(t)|^{2}.$$

Due to (\ref{est:F+equiv}), this estimate implies (\ref{th:unstable+}).

\paragraph{\textmd{\textit{Estimate on the first component}}}

We show that 
\begin{equation}
\left(|u_{-}'(t)|^{2}+|Au_{-}(t)|^{2}\right)\leq c_{2}\limsup_{t\to +\infty}|f(t)|^{2}.
\label{th:unstable-}
\end{equation}

To this end, we rewrite (\ref{eqn:u-}) in the form
$$u_{-}''+u_{-}'-\lambda_{1}(\lambda-\lambda_{1})u_{-}=-\lambda_{1}^{2}(u_{-})^{3}-\lambda_{1}|A^{1/2}u_{+}|^{2} u_{-}+f_{-}(t),$$
and we interpret it as a non-homogeneous linear equation with a prescribed right-hand side. This equation fits in the framework of Lemma~\ref{lemma:ODE-limsup} with $y(t):=u_{-}(t)$, $m:=\lambda_{1}(\lambda-\lambda_{1})$, and
$$\psi(t):=-\lambda_{1}^{2}[u_{-}(t)]^{3}-\lambda_{1}|A^{1/2}u_{+}(t)|^{2}\cdot u_{-}(t)+f_{-}(t).$$

Indeed, the solution $u_{-}(t)$ is bounded because the $\limsup$ of $|u_{-}(t)|$ is bounded, and the forcing term $\psi(t)$ is bounded because $u_{-}(t)$ and $f_{-}(t)$ are bounded, and $|A^{1/2}u_{+}(t)|$ is bounded  because of (\ref{th:unstable+}) and the coercivity of $A$. Thus from conclusion (\ref{th:lemma-ODE-u}) of Lemma~\ref{lemma:ODE-limsup} it turns out that
\begin{equation}
\limsup_{t\to +\infty}|u_{-}(t)|\leq\frac{1}{\lambda_{1}(\lambda-\lambda_{1})}\limsup_{t\to +\infty}|\psi(t)|.
\label{est:limsup-u-psi}
\end{equation}

In order to estimate the right-hand side, we observe that
$$|\psi(t)| \leq \lambda_{1}^{2}|u_{-}(t)|^{2}\cdot|u_{-}(t)|+\lambda_{1}|u_{-}(t)|\cdot|A^{1/2}u_{+}(t)|^{2}+|f_{-}(t)|.$$

When we pass to the $\limsup$, we use again (\ref{th:unstable+}) and the coercivity of $A$, and we deduce that
\begin{eqnarray}
\limsup_{t\to +\infty}|\psi(t)|  &  \leq  & \lambda_{1}^{2}\beta_{0}^{2}\cdot\limsup_{t\to +\infty}|u_{-}(t)|+\lambda_{1}\beta_{0}\cdot c_{3}\limsup_{t\to +\infty}|f_{+}(t)|^{2}+\limsup_{t\to +\infty}|f_{-}(t)|  
\nonumber \\
\noalign{\vspace{1ex}}
 & \leq & \frac{\lambda_{1}(\lambda-\lambda_{1})}{2}\cdot\limsup_{t\to +\infty}|u_{-}(t)|+c_{4}\limsup_{t\to +\infty}|f(t)|,
\label{est:psi-beta0}
\end{eqnarray}
where in the last step we exploited the second inequality in (\ref{defn:beta0}), and our assumption that the $\limsup$ of $|f(t)|$ is less than or equal to~1. 

Plugging this estimate into (\ref{est:limsup-u-psi}), we conclude that
\begin{equation}
\limsup_{t\to +\infty}|u_{-}(t)|\leq c_{5}\limsup_{t\to +\infty}|f(t)|.
\label{limsup:u-}
\end{equation} 

Similarly, from conclusion (\ref{th:lemma-ODE-u'}) of Lemma~\ref{lemma:ODE-limsup} we deduce that
$$\limsup_{t\to +\infty}|u_{-}'(t)|\leq 2\limsup_{t\to +\infty}|\psi(t)|.$$

Keeping (\ref{est:psi-beta0}) and (\ref{limsup:u-}) into account, we conclude that
\begin{equation}
\limsup_{t\to +\infty}|u_{-}'(t)|\leq c_{6}\limsup_{t\to +\infty}|f(t)|.
\label{limsup:u'-}
\end{equation} 

At this point, (\ref{limsup:u-}) and (\ref{limsup:u'-}) imply (\ref{th:unstable-}), which together with (\ref{th:unstable+}) implies the conclusion in (\ref{th:unstable}).\qed

\subsection{Proof of Proposition~\ref{prop:stable}}

\paragraph{\textmd{\textit{Double well potential}}}

Let us introduce the double well potential defined by
$$W(x):=\frac{\lambda_{1}^{2}}{4}\left(x^{2}-\sigma_{0}^{2}\right)^{2}
\quad\quad
\forall x\in\re,$$
where $\sigma_{0}$ is the constant defined in (\ref{defn:sigma-0}). We observe that equation (\ref{eqn:u-}) can now be written in the form
\begin{equation}
u_{-}''+u_{-}'+W'(u_{-})+\lambda_{1}|A^{1/2}u_{+}|^{2}u_{-}=f_{-}(t),
\nonumber
\end{equation}
and that the energy $E_{-}(t)$ of the first component defined in (\ref{defn:E-}) can now be written as
$$E_{-}(t)=\frac{1}{2}|u_{-}'(t)|^{2}+W(u_{-}(t))-W(0).$$

\paragraph{\textmd{\textit{Choice of parameters}}}

Given any $\beta\in(0,\sigma_{0})$ as in the statement of the proposition, let us choose $\beta_{1}\in(0,\beta)$ such that
\begin{equation}
\delta:=W(\beta_{1})-W(0)+\frac{\gamma_{0}}{6}\beta^{2}>0,
\label{defn:delta}
\end{equation}
and let us observe that there exist three positive constants $K_{1}$, $K_{2}$, and $K_{3}$ such that
\begin{eqnarray}
& W(x)\geq W(0)+1
\quad\quad
\forall x\geq K_{1}, &
\label{defn:K1} \\[1ex]
& W(x)\leq K_{2}(x-\sigma_{0})^{2}
\quad\quad
\forall x\in[0,\sqrt{2}\,\sigma_{0}], &
\label{defn:K2} \\[1ex]
& (x-\sigma_{0})\cdot W'(x)\geq K_{3}(x-\sigma_{0})^{2}
\quad\quad
\forall x\geq\beta_{1}. 
\label{defn:K3}
\end{eqnarray}

Let $\gamma_{1}$ be a positive real number satisfying the following five inequalities
\begin{equation}
\gamma_{1}\leq\frac{1}{4},
\quad\quad
K_{2}\gamma_{1}+\frac{\gamma_{1}}{4}+\gamma_{1}^{2}\leq K_{3},
\quad\quad
\lambda_{1}\gamma_{1}\sqrt{2}\,\sigma_{0}^{2}\leq(\lambda_{2}-\lambda)\gamma_{0},
\label{defn:gamma1}
\end{equation}
\begin{equation}
\gamma_{1}^{2}\leq\gamma_{0},
\hspace{3em}
2(K_{1}+\sigma_{0})^{2}\gamma_{1}\leq\delta.
\label{defn:gamma1-bis}
\end{equation}

Let $\eta$ be a positive real number satisfying the following two inequalities
\begin{equation}
\eta\leq 1,
\hspace{3em}
\left[1+\left(\frac{\gamma_{0}}{3}+\gamma_{1}\right)(K_{1}+\sigma_{0})\right]\eta\leq\frac{\delta}{4}.
\label{defn:eta}
\end{equation}

Let $\ep_{1}$ be a positive real number such that
\begin{equation}
\ep_{1}^{2}<\frac{\delta}{4}\gamma_{1}^{2}.
\label{defn:ep1}
\end{equation}

We claim that implication (\ref{th:stable}) holds true for this choice of the parameters.

\paragraph{\textmd{\textit{Definition of energies}}}

For every solution $u(t)$ to (\ref{eqn:duffing}), let us consider the energy
\begin{equation}
R(t):=\frac{1}{2}|u_{-}'(t)|^{2}+W(u_{-}(t))+\gamma_{1}(u_{-}(t)-\sigma_{0})\cdot u_{-}'(t)+\frac{\gamma_{1}}{2}|u_{-}(t)-\sigma_{0}|^{2},
\label{defn:R}
\end{equation}
depending only on the first component, and the global energy
\begin{equation}
S(t):=R(t)+F_{+}(t)+I(t),
\label{defn:S}
\end{equation}
where $F_{+}(t)$ and $I(t)$ are defined in (\ref{defn:F+}) and (\ref{defn:I}).

We show that $F(t)$ estimates the first component $u_{-}(t)$ in the sense that
\begin{equation}
W(u_{-}(t))\leq F(t)+W(0)
\quad\quad
\forall t\geq 0,
\label{est:F-W}
\end{equation}
that $R(t)$ estimates the first component $u_{-}(t)$ in the sense that
\begin{equation}
R(t)\geq\frac{\gamma_{1}}{4}\left(|u_{-}'(t)|^{2}+|u_{-}(t)-\sigma_{0}|^{2}\right)
\quad\quad
\forall t\geq 0,
\label{est:R-u-}
\end{equation}
\begin{equation}
R(t)\geq W(u_{-}(t))
\quad\quad
\forall t\geq 0,
\label{est:R-W}
\end{equation}
and that the global energy $S(t)$ estimates the distance between the whole solution $u(t)$ and the stationary point $\sigma_{0}e_{1}$ in the sense that there exists a positive constant $c_{7}$ such that
\begin{equation}
S(t)\geq c_{7}\left(|u'(t)|^{2}+|A(u(t)-\sigma_{0}e_{1})|^{2}\right)
\quad\quad
\forall t\geq 0.
\label{est:S-energy}
\end{equation}

In order to prove (\ref{est:F-W}), we just observe that the terms $F_{+}(t)$ and $I(t)$ in the definition of $F(t)$ are nonnegative, and therefore form (\ref{est:F-equiv}) it follows that
$$F(t)\geq F_{-}(t)\geq\frac{\lambda_{1}^{2}}{4}|u_{-}(t)|^{4}-\frac{\lambda_{1}(\lambda-\lambda_{1})}{2}|u_{-}(t)|^{2}=W(u_{-}(t))-W(0).$$

In order to prove (\ref{est:R-u-}) and (\ref{est:R-W}), it is enough to plug the inequality
$$|\gamma_{1}(u_{-}(t)-\sigma_{0})\cdot u_{-}'(t)|\leq\frac{\gamma_{1}}{4}|u_{-}(t)-\sigma_{0}|^{2}+\gamma_{1}|u_{-}'(t)|^{2}$$
into (\ref{defn:R}) in order to obtain that
$$R(t)\geq\left(\frac{1}{2}-\gamma_{1}\right)|u_{-}'(t)|^{2}+\frac{\gamma_{1}}{4}|u_{-}(t)-\sigma_{0}|^{2}+W(u_{-}(t)).$$

Since $\gamma_{1}\leq 1/4$ and $W(u_{-}(t))\geq 0$, this proves both (\ref{est:R-u-}) and (\ref{est:R-W}).

Finally, in order to prove (\ref{est:S-energy}), it is enough to consider  (\ref{defn:S}) and recall that
\begin{itemize}
  \item $R(t)$ controls $u_{-}(t)-\sigma_{0}$ and its time-derivative because of (\ref{est:R-u-}),
  \item $F_{+}(t)$ controls $u_{+}(t)$ and its time-derivative because of (\ref{est:F+equiv}),
  \item $I(t)$ is nonnegative.
\end{itemize}

\paragraph{\textmd{\textit{Global energy and potential well}}}

We show that, for every $t\geq 0$, the following implication holds true:
\begin{equation}
\fbox{$S(t)\leq W(\beta_{1})$ and $u_{-}(t)\geq 0$}
\quad\Longrightarrow\quad
\fbox{$\beta_{1}\leq u_{-}(t)\leq\sqrt{2}\sigma_{0}$}
\label{th:L2->well}
\end{equation} 

Indeed, since $F_{+}(t)$ and $I(t)$ are nonnegative, from (\ref{defn:S}) and (\ref{est:R-W}) it follows that
$$W(u_{-}(t))\leq R(t)\leq S(t)\leq W(\beta_{1}).$$

When $u_{-}(t)$ is nonnegative, this inequality implies that $u_{-}(t)$ lies between $\beta_{1}$ and a positive number less than $\sqrt{2}\sigma_{0}$.

\paragraph{\textmd{\textit{Energy estimate at the initial time}}}

Let $T_{0}\geq 0$ be the time mentioned in the assumptions of (\ref{th:stable}). We show that
\begin{equation}
S(T_{0})\leq W(\beta_{1})-\frac{\delta}{2}.
\label{est:S0}
\end{equation}

To this end, we first observe that $S(t)$ and $F(t)$ are related by the equality
\begin{eqnarray*}
S(t) & = & W(0)+F(t)-\frac{\gamma_{0}}{3} u_{-}(t)u_{-}'(t)-\frac{\gamma_{0}}{6}|u_{-}(t)|^{2} \\[1ex]
 &  & \mbox{}+\gamma_{1}(u_{-}(t)-\sigma_{0})\cdot u_{-}'(t)+\frac{\gamma_{1}}{2}|u_{-}(t)-\sigma_{0}|^{2},
\end{eqnarray*}
from which it follows that
\begin{eqnarray}
S(t) & \leq & W(0)+F(t)-\frac{\gamma_{0}}{6}|u_{-}(t)|^{2}
\nonumber \\[1ex]
 &  &  \mbox{}+\left(\frac{\gamma_{0}}{3}+\gamma_{1}\right)\left(|u_{-}(t)|+\sigma_{0}\strut\right)\cdot|u_{-}'(t)|+\frac{\gamma_{1}}{2}\left(|u_{-}(t)|+\sigma_{0}\strut\right)^{2} 
\label{est:S-F}
\end{eqnarray}
for every $t\geq 0$. In order to estimate all these terms, we exploit our assumptions that $F(T_{0})<\eta$, $|u_{-}'(T_{0})|<\eta$, and $u_{-}(T_{0})>\beta$. We need also an estimate from above for $u_{-}(T_{0})$, which we deduce again from the assumption that $F(T_{0})<\eta$. Indeed, since $\eta\leq 1$, from (\ref{est:F-W}) we deduce that $W(u_{-}(T_{0}))\leq W(0)+1$. Thanks to (\ref{defn:K1}), this inequality implies that $u_{-}(T_{0})\leq K_{1}$.

Plugging all these estimates into (\ref{est:S-F}), and keeping the second smallness assumptions in (\ref{defn:gamma1-bis}) and (\ref{defn:eta}) into account, we conclude that
\begin{eqnarray*}
S(T_{0}) & \leq &  W(0)+\eta-\frac{\gamma_{0}}{6}\beta^{2}+\left(\frac{\gamma_{0}}{3}+\gamma_{1}\right)(K_{1}+\sigma_{0})\eta+\frac{\gamma_{1}}{2}(K_{1}+\sigma_{0})^{2} \\[1ex]
 & \leq  &  W(0)-\frac{\gamma_{0}}{6}\beta^{2}+\frac{\delta}{2},
\end{eqnarray*}
which is exactly (\ref{est:S0}) when $\delta$ is given by (\ref{defn:delta}).

\paragraph{\textmd{\textit{Energy estimate in the potential well}}}

We show that the time-derivatives of $R(t)$ and $S(t)$ satisfy the following two inequalities
\begin{equation}
R'(t)\leq-\gamma_{1}^{2}R(t)+2|f_{-}(t)|^{2}-\lambda_{1}|A^{1/2}u_{+}(t)|^{2}u_{-}(t)u_{-}'(t)+\frac{\lambda_{2}-\lambda}{\lambda_{2}}\gamma_{0}|Au_{+}(t)|^{2},
\label{est:R'}
\end{equation}
\begin{equation}
S'(t)\leq -\gamma_{1}^{2}S(t)+2|f(t)|^{2},
\label{est:S'}
\end{equation}
\emph{as long as the solution lies in the potential well}, more precisely for every $t\geq 0$ such that
\begin{equation}
\beta_{1}\leq u_{-}(t)\leq\sqrt{2}\sigma_{0}.
\label{hp:ALA}
\end{equation}

To this end, with some algebra we write the time-derivative of $R(t)$ in the form
\begin{eqnarray}
R'(t) & = & -\gamma_{1}^{2}R(t)-\left(1-\gamma_{1}-\frac{\gamma_{1}^{2}}{2}\right)|u_{-}'(t)|^{2}+\frac{\gamma_{1}^{3}}{2}|u_{-}(t)-\sigma_{0}|^{2}  
\nonumber \\
\noalign{\vspace{1ex}}
 &  &  \mbox{}-\gamma_{1}(u_{-}(t)-\sigma_{0})\cdot W'(u_{-}(t))+\gamma_{1}^{2}W(u_{-}(t)) 
\nonumber  \\
\noalign{\vspace{1ex}}
 &  & \mbox{}-\lambda_{1}|A^{1/2}u_{+}(t)|^{2}u_{-}(t)u_{-}'(t)-\gamma_{1}\lambda_{1}|A^{1/2}u_{+}(t)|^{2}(u_{-}(t)-\sigma_{0})u_{-}(t) 
\nonumber  \\
\noalign{\vspace{1ex}}
 &  &  \mbox{}+u_{-}'(t)f_{-}(t) +\gamma_{1}(u_{-}(t)-\sigma_{0})f_{-}(t) +\gamma_{1}^{3}u_{-}'(t)(u_{-}(t)-\sigma_{0}) .
 \label{eqn:R'}
\end{eqnarray}

The terms in the last line can be estimated as follows
\begin{eqnarray*}
 & \displaystyle|u_{-}'(t)\cdot f_{-}(t)|\leq\frac{1}{4}|u_{-}'(t)|^{2}+|f_{-}(t)|^{2}, & \\[1ex]
 & \displaystyle|\gamma_{1}(u_{-}(t)-\sigma_{0})\cdot f_{-}(t)|\leq\frac{\gamma_{1}^{2}}{4}|u_{-}(t)-\sigma_{0}|^{2}+|f_{-}(t)|^{2}, & \\[1ex]
 & \displaystyle|\gamma_{1}^{3}u_{-}'(t)\cdot (u_{-}(t)-\sigma_{0})|\leq\frac{\gamma_{1}^{3}}{2}|u_{-}'(t)|^{2}+\frac{\gamma_{1}^{3}}{2}|u_{-}(t)-\sigma_{0}|^{2}. & 
\end{eqnarray*}

Moreover, from (\ref{hp:ALA}) we deduce that $|u_{-}(t)-\sigma_{0}|\leq\sigma_{0}$, and therefore
$$\gamma_{1}\lambda_{1}|A^{1/2}u_{+}(t)|^{2}\cdot|u_{-}(t)-\sigma_{0}|\cdot|u_{-}(t)|\leq\frac{\gamma_{1}\lambda_{1}}{\lambda_{2}}|Au_{+}(t)|^{2}\cdot\sigma_{0}\cdot\sqrt{2}\sigma_{0}.$$

Finally, from (\ref{defn:K2}) and (\ref{defn:K3}) (in this point we need again (\ref{hp:ALA})) we deduce that
$$-\gamma_{1}(u_{-}(t)-\sigma_{0})\cdot W'(u_{-}(t))+\gamma_{1}^{2}W(u_{-}(t))\leq\left(-\gamma_{1}K_{3}+\gamma_{1}^{2}K_{2}\right)|u_{-}(t)-\sigma_{0}|^{2}.$$

Plugging all these estimates into (\ref{eqn:R'}) we obtain that
\begin{eqnarray*}
R'(t) & \leq & -\gamma_{1}^{2}R(t)-\left(\frac{3}{4}-\gamma_{1}-\frac{\gamma_{1}^{2}}{2}-\frac{\gamma_{1}^{3}}{2}\right)|u_{-}'(t)|^{2} \\
 &  & \mbox{}-\gamma_{1}\left(K_{3}-K_{2}\gamma_{1}-\frac{\gamma_{1}}{4}-\gamma_{1}^{2}\right)|u_{-}(t)-\sigma_{0}|^{2}  \\
 &  &  \mbox{}+2|f_{-}(t)|^{2}-\lambda_{1}|A^{1/2}u_{+}(t)|^{2}u_{-}(t)u_{-}'(t)+\frac{\gamma_{1}\lambda_{1}}{\lambda_{2}}\sqrt{2}\,\sigma_{0}^{2}|Au_{+}(t)|^{2}
\end{eqnarray*}
as long as condition (\ref{hp:ALA}) is satisfied.
At this point, (\ref{est:R'}) follows from the three smallness conditions on $\gamma_{1}$ stated in (\ref{defn:gamma1}).

Finally, since $\gamma_{1}^{2}\leq\gamma_{0}$, estimate (\ref{est:S'}) follows from (\ref{est:F+'}) and (\ref{eqn:I'}), which hold true for every $t\geq 0$, and (\ref{est:R'}), which holds true as long as the solution satisfies (\ref{hp:ALA}).

\paragraph{\textmd{\textit{Solutions remain in the potential well}}}

We show that
\begin{equation}
S(t)\leq W(\beta_{1})\mbox{ and }u_{-}(t)\geq 0 
\quad\quad
\forall t\geq T_{0}.
\label{th:u-pot-well}
\end{equation}

To this end, let us set 
$$T_{1}:=\sup\left\{t\geq T_{0}:S(\tau)< W(\beta_{1})\mbox{ and }u_{-}(\tau)> 0\mbox{ for every }\tau\in[T_{0},t]\strut\right\}.$$

We observe that $T_{1}$ is the supremum of an open set containing $t=T_{0}$, and hence it is well defined and greater than $T_{0}$, and it satisfies
\begin{equation}
S(t)\leq W(\beta_{1})\quad\mbox{and}\quad u_{-}(t)\geq 0
\quad\quad
\forall t\in[T_{0},T_{1}].
\label{ineq:0-T}
\end{equation}

If $T_{1}=+\infty$, then (\ref{th:u-pot-well}) is proved. Let us assume by contradiction that $T_{1}<+\infty$. Due to the maximality of $T_{1}$, it follows that either $S(T_{1})=W(\beta_{1})$ or  $u_{-}(T_{1})=0$. Now we show that both possibilities lead to an absurdity. 

From (\ref{ineq:0-T}) and (\ref{th:L2->well}) it follows that (\ref{hp:ALA}) holds true for every $t\in[T_{0},T_{1}]$, and hence also the differential inequality (\ref{est:S'}) is satisfied for every $t\in[T_{0},T_{1}]$. Recalling that $|f(t)|\leq\ep_{1}$ for every $t\geq T_{0}$, integrating this differential inequality we deduce that
$$S(t)\leq S(T_{0})+\frac{2\ep_{1}^{2}}{\gamma_{1}^{2}}
\qquad
\forall t\in[T_{0},T_{1}].$$

Setting $t=T_{1}$, and keeping into account (\ref{est:S0}) and the smallness condition (\ref{defn:ep1}), we conclude that $S(T_{1})<W(\beta_{1})$, which rules out the first possibility.

On the other hand, we already know from (\ref{th:L2->well}) that the two inequalities $S(T_{1})\leq W(\beta_{1})$ and $u_{-}(T_{1})\geq 0$ imply that $u_{-}(T_{1})\geq\beta_{1}$, thus ruling out the possibility that $u_{-}(T_{1})=0$.

\paragraph{\textmd{\textit{Conclusion}}}

From (\ref{th:u-pot-well}) and (\ref{th:L2->well}) we deduce that the potential well assumption (\ref{hp:ALA}) holds true for every $t\geq T_{0}$, and hence also the differential inequality (\ref{est:S'}) is now satisfied for every $t\geq T_{0}$. Integrating this differential inequality we obtain that
$$\limsup_{t\to+\infty}S(t)\leq\frac{2}{\gamma_{1}^{2}}\limsup_{t\to+\infty}|f(t)|^{2},$$
which is  equivalent to the conclusion in (\ref{th:stable}) because of (\ref{est:S-energy}).\qed

\subsection{Proof of Proposition~\ref{prop:asymptotic}}

Since $u(t)$ and $v(t)$ are solutions to (\ref{eqn:duffing}), their difference $r(t):=u(t)-v(t)$ satisfies
\begin{equation}
r''(t)+r'(t)+A^{2}r(t)-\lambda Ar(t)=g(t),
\label{eqn:r}
\end{equation}
where
\begin{eqnarray}
g(t) & := & -|A^{1/2}u(t)|^{2}Au(t)+|A^{1/2}v(t)|^{2}Av(t) \nonumber \\
 & = & -|A^{1/2}u(t)|^{2}Ar(t)-\langle u(t)+v(t),Ar(t)\rangle Av(t).
 \label{eqn:g-r}
\end{eqnarray}

This equation requires a separate treatment in the unstable case $\sigma=0$ and in the stable cases $\sigma=\pm\sigma_{0}$. The constants $c_{8}$, \ldots, $c_{22}$ in the sequel depend only on the three parameters $\lambda$, $\lambda_{1}$, $\lambda_{2}$.

\paragraph{\textmd{\textit{Unstable case}}}

From (\ref{eqn:g-r}) we deduce that
\begin{equation}
|g(t)|\leq\left(|A^{1/2}u(t)|^{2}+|u(t)+v(t)|\cdot|Av(t)|\right)|Ar(t)|.
\label{est:g-u}
\end{equation}
On the other hand, from (\ref{hp:asymptotic}) with $\sigma=0$ and the coerciveness of $A$ we know that
$$\limsup_{t\to+\infty}|A^{1/2}u(t)|^{2}\leq c_{8}r_{0}^{2}
\quad\quad\mbox{and}\quad\quad
\limsup_{t\to+\infty}|u(t)+v(t)|\cdot|Av(t)|\leq c_{9}r_{0}^{2},$$
and therefore from  (\ref{est:g-u}) we obtain that
\begin{equation}
\limsup_{t\to+\infty}|g(t)|\leq c_{10}r_{0}^{2}\cdot\limsup_{t\to+\infty}|Ar(t)|.
\label{est:gAr}
\end{equation}

Now let us write as usual $r(t)=r_{-}(t)e_{1}+r_{+}(t)$ and $g(t)=g_{-}(t)e_{1}+g_{+}(t)$, so that equation (\ref{eqn:r}) is equivalent to the system
\begin{equation}
r_{-}''(t)+r_{-}'(t)-\lambda_{1}(\lambda-\lambda_{1})r_{-}(t)=g_{-}(t),
\label{eqn:r-}
\end{equation}
\begin{equation}
r''_{+}(t)+r'_{+}(t)+A^{2}r_{+}(t)-\lambda Ar_{+}(t)=g_{+}(t).
\label{eqn:r+}
\end{equation}

Equation (\ref{eqn:r-}) is a scalar equation that fits in the framework of Lemma~\ref{lemma:ODE-limsup} with
$$y(t):=r_{-}(t),
\quad\quad
m:=\lambda_{1}(\lambda-\lambda_{1}),
\quad\quad
\psi(t):=g_{-}(t).$$

Indeed, $r_{-}(t)$ is bounded because $u(t)$ and $v(t)$ are bounded, and for the same reason also $g_{-}(t)$ is bounded. As a consequence, from  Lemma~\ref{lemma:ODE-limsup} we deduce that
\begin{equation}
\limsup_{t\to+\infty}\left(|r_{-}'(t)|^{2}+|Ar_{-}(t)|^{2}\right)\leq c_{11}\limsup_{t\to+\infty}|g_{-}(t)|^{2}.
\label{est:r-g-}
\end{equation}

Equation (\ref{eqn:r+}) fits in the  framework of Lemma~\ref{lemma:PDE-limsup} with
$$X:=H_{+},
\quad\quad
B:=A^{2}-\lambda A,
\quad\quad
m:=\lambda_{2}(\lambda_{2}-\lambda),
\quad\quad
\psi(t):=g_{+}(t).$$

In addition, there exists a constant $c_{12}$ such that
$$|Ax|^{2}\leq c_{12}|B^{1/2}x|^{2}
\quad\quad
\forall x\in D(A)\cap H_{+}.$$

As a consequence, from Lemma~\ref{lemma:PDE-limsup} we deduce that
\begin{eqnarray}
\limsup_{t\to+\infty}\left(|r_{+}'(t)|^{2}+|Ar_{+}(t)|^{2}\right) & \leq & c_{13}\limsup_{t\to+\infty}\left(|r_{+}'(t)|^{2}+|B^{1/2}r_{+}(t)|^{2}\right) 
\nonumber \\
 & \leq & c_{14}\limsup_{t\to+\infty}|g_{+}(t)|^{2}.
 \label{est:r+g+}
\end{eqnarray}

From (\ref{est:r-g-}), (\ref{est:r+g+}), and (\ref{est:gAr}) we conclude that
\begin{eqnarray*}
\limsup_{t\to+\infty}\left(|r'(t)|^{2}+|Ar(t)|^{2}\right) & \leq & 
\limsup_{t\to+\infty}\left(|r_{-}'(t)|^{2}+|Ar_{-}(t)|^{2}\right) \\
 & & \mbox{} +\limsup_{t\to+\infty}\left(|r_{+}'(t)|^{2}+|Ar_{+}(t)|^{2}\right) \\
 & \leq & c_{15}\limsup_{t\to+\infty}|g(t)|^{2} \\
 & \leq & c_{16}r_{0}^{4}\cdot\limsup_{t\to+\infty}|Ar(t)|^{2} \\
 & \leq & c_{16}r_{0}^{4}\cdot\limsup_{t\to+\infty}\left(|r'(t)|^{2}+|Ar(t)|^{2}\right).
\end{eqnarray*}

If $r_{0}$ is small enough, the coefficient of the last $\limsup$ is less than~1. It follows that
\begin{equation}
\lim_{t\to+\infty}\left(|r'(t)|^{2}+|Ar(t)|^{2}\right)=0,
\label{th:lim-r}
\end{equation}
which in turn is equivalent to (\ref{th:main-asymptotic}).

\paragraph{\textmd{\textit{Stable case}}}

We assume, without loss of generality, that $\sigma=\sigma_{0}$ (the case $\sigma=-\sigma_{0}$ being symmetric). In order to exploit the smallness of $u(t)-\sigma_{0}e_{1}$ and $v(t)-\sigma_{0}e_{1}$, with some algebra we rewrite (\ref{eqn:g-r}) in the form
\begin{equation}
g(t)=-\sigma_{0}^{2}\lambda_{1}Ar(t)-2\sigma_{0}^{2}\lambda_{1}\langle Ar(t),e_{1}\rangle e_{1}+\hg(t),
\label{eqn:g-stable}
\end{equation}
where
\begin{eqnarray*}
\hg(t) & := & -\left(|A^{1/2}(u(t)-\sigma_{0}e_{1})|^{2}+
2\langle A(u(t)-\sigma_{0}e_{1}),\sigma_{0}e_{1}\rangle\right)Ar(t) \\
 & & -\langle u(t)+v(t),Ar(t)\rangle A(v(t)-\sigma_{0}e_{1}) \\
 & & -\langle u(t)+v(t)-2\sigma_{0}e_{1},Ar(t)\rangle\lambda_{1}\sigma_{0}e_{1}.
\end{eqnarray*}

Therefore from (\ref{hp:asymptotic}) with $\sigma=\sigma_{0}$, we deduce that
\begin{equation}
\limsup_{t\to+\infty}|\hg(t)|\leq
\left(c_{17}r_{0}+c_{18}r_{0}^{2}\right)\limsup_{t\to+\infty}|Ar(t)|.
\label{est:g123}
\end{equation}

Plugging (\ref{eqn:g-stable}) into (\ref{eqn:r}), we obtain that $r(t)$ is a solution to
$$r''(t)+r'(t)+A^{2}r(t)-\lambda Ar(t)+\sigma_{0}^{2}\lambda_{1}Ar(t)+2\sigma_{0}^{2}\lambda_{1}\langle Ar(t),e_{1}\rangle e_{1}=\hg(t).$$

Keeping (\ref{defn:sigma-0}) into account, this equation can be rewritten as
\begin{equation}
r''(t)+r'(t)+Lr(t)=\hg(t),
\label{eqn:r-stable}
\end{equation}
where $L$ is the linear operator on $H$ defined by 
$$Lx=\left\{\!
\begin{array}{l@{\quad\quad}l}
A^{2}x-\lambda_{1}Ax & \mbox{if }x\in D(A^{2})\cap H_{+},  \\
\noalign{\vspace{1ex}}
2\lambda_{1}(\lambda-\lambda_{1})e_{1} & \mbox{if }x=e_{1}.
\end{array}\right.$$

This operator is coercive, more precisely
$$\langle Lx,x\rangle\geq\min\left\{2\lambda_{1}(\lambda-\lambda_{1}),\lambda_{2}(\lambda_{2}-\lambda_{1})\strut\right\}|x|^{2}=:m_{0}|x|^{2}
\qquad
\forall x\in D(A),$$ 
and therefore (\ref{eqn:r-stable}) fits in the framework of Lemma~\ref{lemma:PDE-limsup} with
$$X:=H,
\quad\quad
B:=L,
\qquad
y(t):=r(t),
\quad\quad
m:=m_{0},
\quad\quad
\psi(t):=\hg(t).$$

In addition, there exists a constant $c_{19}$ such that
$$|Ax|^{2}\leq c_{19}|B^{1/2}x|^{2}
\quad\quad
\forall x\in D(A).$$

As a consequence, from Lemma~\ref{lemma:PDE-limsup} we deduce that
\begin{eqnarray}
\limsup_{t\to+\infty}\left(|r'(t)|^{2}+|Ar(t)|^{2}\right) & \leq & c_{20}\limsup_{t\to+\infty}\left(|r'(t)|^{2}+|B^{1/2}r(t)|^{2}\right) 
\nonumber \\
 & \leq & c_{21}\limsup_{t\to+\infty}|\hg(t)|^{2}.
\nonumber
\end{eqnarray}

Keeping (\ref{est:g123}) into account, we can continue this chain of inequalities, and obtain that
\begin{eqnarray*}
\limsup_{t\to+\infty}\left(|r'(t)|^{2}+|Ar(t)|^{2}\right) & \leq & 
c_{21}\left(c_{17}r_{0}+c_{18}r_{0}^{2}\right)^{2}\cdot\limsup_{t\to+\infty}|Ar(t)|^{2} 
\\
 & \leq & c_{22}(r_{0}^{2}+r_{0}^{4})\cdot\limsup_{t\to+\infty}\left(|r'(t)|^{2}+|Ar(t)|^{2}\right).
\end{eqnarray*}

If $r_{0}$ is small enough, we obtain again (\ref{th:lim-r}), which in turn is equivalent to (\ref{th:main-asymptotic}).\qed

\subsubsection*{\centering Acknowledgments}

The first two authors are members of the ``Gruppo Nazionale per l'Analisi Matematica, la Probabilit\`{a} e le loro Applicazioni'' (GNAMPA) of the ``Istituto Nazionale di Alta Matematica'' (INdAM). 


\label{NumeroPagine}

\end{document}